\documentclass[12pt]{article}
\usepackage{amsmath,amssymb,amsthm,bm}
\usepackage{enumitem}
\usepackage{color}
\usepackage{mathtools}

\newtheorem{dfe}{Definition}
[section]
\newtheorem{lem}[dfe]{Lemma}
\newtheorem{thm}[dfe]{Theorem}

\newtheorem{prop}[dfe]{Proposition}
\newtheorem{cor}[dfe]{Corollary}

\newtheorem{Rem}{Remark}

\makeatletter

\@addtoreset{equation}{section}
\makeatother



\newcommand{\C}{\mathbb{C}}

\newcommand{\Z}{\mathbb{Z}}
\newcommand{\N}{\mathbb{N}}
\newcommand{\F}{\mathbb{F}}
\newcommand{\Span}{\mathop{\mathrm{Span}}\nolimits}

\newcommand{\End}{\mathop{\mathrm{End}}\nolimits}

\newcommand{\qbinom}[2]{\genfrac{[}{]}{0pt}{}{#1}{#2}} 
\newcommand{\rank}{\mathop{\mathrm{rank}}\nolimits}

\DeclarePairedDelimiter{\set}{\lbrace}{\rbrace}
\DeclarePairedDelimiter{\card}{\lvert}{\rvert}

\newcommand{\cB}{\mathcal{B}}

\newcommand{\alphaalt}{s}
\newcommand{\betaalt}{t}

\begin{title}
{Bivariate $Q$-polynomial structures for the nonbinary Johnson scheme
and the association scheme obtained from attenuated spaces}
\end{title}
\author{
    Eiichi Bannai\thanks{Faculty of Mathematics, Kyushu University (emeritus), Japan},
    Hirotake Kurihara\thanks{Department of Applied Science, Yamaguchi University, 2-16-1 Tokiwadai, Ube 755-8611, Japan}, 
    Da Zhao\thanks{School of Mathematics, East China University of Science and Technology, Shanghai 200237, China}, 
    Yan Zhu\thanks{College of Science, University of Shanghai for Science and Technology, Shanghai
    200093, China}}
\begin{document}
\maketitle

\begin{abstract}
The study of $P$-polynomial association schemes (distance-regular graphs) and 
$Q$-polynomial association schemes, and in particular $P$- and $Q$-polynomial association schemes, has been 
a central theme not only in the theory of association schemes but also in the whole study of algebraic combinatorics in general. 
Leonard's theorem (1982) 
says that the spherical functions (or the character tables) of $P$- and $Q$-polynomial  
association schemes are described by Askey-Wilson orthogonal 
polynomials or their relatives. These polynomials are one-variable 
orthogonal polynomials. It seems that the new attempt to 
define and study higher rank $P$- and $Q$-polynomial association schemes had been hoped for, but had gotten only limited success. 
The first very successful attempt was initiated recently by Bernard-Cramp\'{e}-d'Andecy-Vinet-Zaimi, and then followed by Bannai-Kurihara-Zhao-Zhu. 
The general theory and some explicit examples of families of higher rank (multivariate) $P$- and/or $Q$-polynomial association 
schemes have been obtained there. The main purpose of the present paper is 
to prove that some important families of association schemes are shown to be bivariate $Q$-polynomial. 
Namely, we show that all the nonbinary Johnson association schemes 
and all the attenuated space association schemes are bivariate $Q$-polynomial. 
It should be noted that the parameter restrictions needed in the previous papers are 
completely lifted in this paper. Our proofs are done by explicitly calculating the 
Krein parameters of these association schemes. At the end, we mention some 
speculations and indications of what we can expect in the future study.  
\end{abstract}

\tableofcontents

\section{Introduction}
\label{sec:intro}
The study of $P$-polynomial association schemes (distance-regular graphs) and 
$Q$-polynomial association schemes has been one of the central problems in the study of 
algebraic combinatorics, in particular in the theory of association schemes. The classification of $P$- and $Q$-polynomial association schemes was very much looked for, and their connection with the theory 
of orthogonal polynomials has been very important since it was shown by Leonard~\cite{Leonard1982} that 
the spherical functions (and character tables) of $P$- and $Q$-polynomial 
association schemes were described by using Askey-Wilson orthogonal 
polynomials including their special or limiting cases
(See \cite{BBIT2021,bi,BCN1989,Terwilliger2001,Terwilliger2021}).
Note that 
the orthogonal polynomials attached to $P$- and $Q$-polynomial association 
schemes are one-variable polynomials. It seems that many researchers wanted to define 
higher rank $P$- and $Q$-polynomial association schemes, but 
only some special cases have been studied, see Mizukawa-Tanaka~\cite{MT2004}, 
Gasper-Rahman~\cite{GR2007}, 
Iliev-Terwilliger~\cite{IT2012} and many others.
On the other hand, one-variable Askey-Wilson orthogonal polynomials have 
been generalized to 
multivariable orthogonal polynomials in various ways and also very much studied, purely as orthogonal polynomials. See, for example, many papers mentioned in the Introduction and the References in Bernard-Cramp\'{e}-d'Andecy-Vinet-Zaimi~\cite{bi}.
However, their connection with the theory of association schemes, in particular, higher rank $P$- and $Q$-polynomial association schemes, was not very much revealed, except for the special 
cases studied in the papers referenced above~\cite{MT2004,GR2007,IT2012}, among others.
The reason for this was that 
the good definition of higher rank (i.e., multivariate) $P$-polynomial or $Q$-polynomial 
association schemes were missing and explicit good examples of 
higher rank $P$-polynomial or $Q$-polynomial 
association schemes were generally missing. So, it was a very pleasant 
surprise when we saw the paper by Bernard-Cramp\'{e}-d'Andecy-Vinet-Zaimi~\cite{bi}, where the bivariate $P$-polynomial association schemes (and 
the bivariate $Q$-polynomial association schemes) of type $(\alpha, \beta)$ 
were defined and showed that some good explicit examples do exist. Motivated by 
their paper \cite{bi}, we started to think of this line of study, and we defined 
multivariable $P$-polynomial (and $Q$-polynomial) association schemes 
generalizing their concept 
for any monomial order. Then, the authors \cite{BKZZ} obtained more examples of (family of) 
multivariate $P$-polynomial association schemes and 
also the multivariate $Q$-polynomial association schemes. 

The purpose of the 
present paper is to show that the two important families are actually 
$Q$-polynomial association schemes in the sense of \cite{BKZZ} with respect to the graded lexicographic order. In most cases, we consider multivariate 
$P$-polynomial (or $Q$-polynomial) association schemes with respect to 
the graded lexicographic order.
Namely, the main results of the present paper are 
the following Theorem~\ref{thm:main1} and Theorem~\ref{thm:main2} to be described below. 
It was shown by Cramp\'{e}-Vinet-Zaimi-Zhang~\cite{biQ} that nonbinary Johnson 
association schemes are bivariate $Q$-polynomial in their sense~\cite{bi} if some restriction 
is imposed on the parameters. Our proof is basically along 
their lines and to calculate the Krein parameters very explicitly. 
However, we can completely lift the parameter restrictions in \cite{biQ}.
Another case was on association schemes coming from attenuated 
spaces. The parameters and spherical functions of this family of 
association schemes are obtained by Wang-Guo-Li~\cite{WGL2010} and 
Kurihara~\cite{Kurihara2013} 
already, but to show that these association schemes are actually bivariate $Q$-polynomial is  
highly non-trivial, as we need to compute the Krein parameters explicitly. We show that they are indeed bivariate $Q$-polynomial, again lifting 
the parameter restrictions completely. 

\begin{thm}
    \label{thm:main1}
    The nonbinary Johnson scheme $J_r(n,k)$ is a bivariate $Q$-polynomial association scheme with respect to the graded lexicographic order $\le_{\mathrm{grlex}}$.
\end{thm}

\begin{thm}
    \label{thm:main2}
    The association scheme obtained from attenuated spaces is a bivariate $Q$-polynomial association scheme with respect to the graded lexicographic order $\le_{\mathrm{grlex}}$.
\end{thm}

We hope and expect that the techniques used in the present paper will be applied to 
other families, in particular for association schemes on (not necessarily maximal) 
isotropic spaces (see Stanton~\cite{Stanton1980}) and generalized Johnson association schemes 
in the sense of Ceccherini-Silberstein, Scarabotti and Filippo~\cite{CST2006}, but the technical difficulties 
are so far beyond our ability.
We want to come back to this question in the near future.
Finally, we note that many examples of our 
multivariate $P$- and $Q$-polynomial association schemes actually 
become some explicit examples of Problem 7.1 in Iliev-Terwilliger~\cite{IT2012}. 
Therefore, we believe this will pave the way for research in new directions.

Now we will provide further details regarding the content of this paper.
In Section~\ref{sec:2}, we will review the definitions of multivariate $P$- or $Q$-polynomial association schemes in the sense of \cite{BKZZ}, 
building upon the work of Bernard et al.~\cite{bi}.
Section~\ref{sec:nonbinary} will present relevant facts concerning nonbinary Johnson schemes.
The proof of Theorem~\ref{thm:main1} will be provided in Section~\ref{sec:proofThm1}.
Section~\ref{sec:attenuated} will discuss association schemes obtained from attenuated spaces.
The proof of Theorem~\ref{thm:main2} will be presented in Section~\ref{sec:proofThm2}.
Furthermore, in Section~\ref{sec:multiLeonard}, we will explore the relationship between nonbinary Johnson schemes, association schemes derived from attenuated spaces, and $A_2$-Leonard pairs.

\section{Multivariate $P$-polynomial and/or $Q$-polynomial 
association schemes}
\label{sec:2}

In this section, we briefly review the definition and properties of multivariate $P$-polynomial and/or $Q$-polynomial association schemes introduced by Bannai-Kurihara-Zhao-Zhu~\cite{BKZZ}.
Please refer to \cite{BKZZ} for more details.
Note that the definition of bivariate $P$-polynomial and/or $Q$-polynomial association schemes was originally introduced by Bernard-Cramp\'{e}-d'Andecy-Vinet-Zaimi~\cite{bi},
and Bannai et al.~\cite{BKZZ} extended it to the multivariate case.

\subsection{Association schemes}
\label{sec:AS}
We begin by recalling the concept of association schemes.
The reader is referred to Bannai-Bannai-Ito-Tanaka~\cite{BBIT2021} and Bannai-Ito~\cite{BI1984} for details.
Let $X$ and $\mathcal{I}$ be finite sets
and let $\mathcal{R}$ be a surjective map from $X\times X$ to $\mathcal{I}$.
For $i\in \mathcal{I}$, we put $R_i=\mathcal{R}^{-1}(i)$,
i.e., $R_i=\{(x,y)\in X\times X\mid \mathcal{R}(x,y)=i\}$.
Let $M_X(\C)$ be the $\C$-algebra of complex matrices with rows and columns indexed by $X$.
The \emph{adjacency matrix} $A_i$ of $i\in \mathcal{I}$
is defined to be the matrix in $M_X(\C)$ whose $(x,y)$ entries are
\[
(A_i)_{xy}=
\begin{cases}
    1 & \text{if $\mathcal{R}(x,y)=i$,}\\   
    0 & \text{otherwise.}
\end{cases}    
\]
It is obvious that
\begin{enumerate}[label=$(\mathrm{A}\arabic*)$]
    \item $\sum_{i\in \mathcal{I}} A_i =J_X$, where $J_X$ is the all-one matrix of $M_X(\C)$. \label{AS:Hadamard}
\end{enumerate}
A pair $\mathfrak{X}=(X,\{R_i\}_{i\in \mathcal{I}})$ (or simply $(X,\mathcal{R})$)
is called a \emph{commutative association scheme} if
$\mathfrak{X}$ satisfies the following conditions:
\begin{enumerate}[label=$(\mathrm{A}\arabic*)$]
    \setcounter{enumi}{1}
    \item there exists $i_0 \in \mathcal{I}$ such that $A_{i_0}=I_X$, where $I_X$ is the identity matrix of $M_X(\C)$; \label{AS:I}
    \item for each $i\in \mathcal{I}$, there exists $i'\in \mathcal{I}$ such that $A_i^T=A_{i'}$, where $A_i^T$ denotes the transpose of $A_i$; \label{AS:transpose}
    \item for each $i,j\in \mathcal{I}$,
    \[
    A_i A_j = \sum_{k\in \mathcal{I}} p^k_{ij} A_k   
    \]
    holds.
    The constant $p^k_{ij}$ is called the \emph{intersection number}; \label{AS:pijk}
    \item for $i,j,k\in \mathcal{I}$, $p^k_{ij}=p^k_{ji}$ holds, i.e., $A_i A_j = A_j A_i$ holds.\label{AS:commutative}
\end{enumerate}
Moreover, if an association scheme $\mathfrak{X}=(X,\{R_i\}_{i\in \mathcal{I}})$ satisfies 
\begin{enumerate}[label=$(\mathrm{A}\arabic*)$]
    \setcounter{enumi}{5}
\item for each $i\in \mathcal{I}$, $i=i'$ holds, \label{AS:symmetric}
\end{enumerate}
then $\mathfrak{X}$ is called \emph{symmetric}.
If the cardinality $|\mathcal{I}|$ of $\mathcal{I}$ is equal to $d+1$,
then $\mathfrak{X}$ is called of class $d$.

We also use the notation $\mathfrak{X}=(X,\{A_i\}_{i\in \mathcal{I}})$
to denote association schemes with the adjacency matrices $\{A_i\}_{i\in \mathcal{I}}$.
Let $\mathfrak{A}=\Span_\C \{ A_i \}_{i\in \mathcal{I}}$.
By \ref{AS:I} and \ref{AS:pijk}, $\mathfrak{A}$ becomes a subalgebra of $M_X(\C)$.
The algebra $\mathfrak{A}$ is called the \emph{Bose-Mesner algebra} of $\mathfrak{X}$.
By \ref{AS:Hadamard}, $\{A_i\}_{i\in \mathcal{I}}$ is a basis of $\mathfrak{A}$,
i.e., $\dim_{\C} \mathfrak{A}=|\mathcal{I}|$.

By \ref{AS:commutative}, $\mathfrak{A}$ has another basis  
$\{E_j\}_{j\in \mathcal{J}}$ consisting of the primitive idempotents of $\mathfrak{A}$,
where $\mathcal{J}$ is a finite set
and there exists $j_0\in \mathcal{J}$ such that $E_{j_0}=\frac{1}{|X|}J_X$.
Since $\{A_i\}_{i\in \mathcal{I}}$ and $\{E_j\}_{j\in \mathcal{J}}$ are bases of $\mathfrak{A}$,
$|\mathcal{I}|=|\mathcal{J}|$ holds.
By \ref{AS:Hadamard}, $\mathfrak{A}$ is closed under entrywise multiplication,
which product is denoted by $\circ$ and called the \emph{Hadamard product}.
Then for $i,j,k\in \mathcal{J}$,
there exists a real number (in fact, a nonnegative number) $q^k_{ij}$
such that $(|X| E_i)\circ (|X|E_j)=\sum_{k\in \mathcal{J}}q^k_{ij}|X|E_k$,
and $q^k_{ij}$ are called the \emph{Krein numbers} of $\mathfrak{X}$.

The entries of the \emph{first eigenmatrix} $P:=(P_i(j))_{j\in \mathcal{J}, i\in \mathcal{I}}$
and the \emph{second eigenmatrix} $Q:=(Q_j(i))_{i\in \mathcal{I}, j\in \mathcal{J}}$ of $\mathfrak{X}$
are defined by
\[
A_i = \sum_{j\in \mathcal{J}}P_i(j)E_j \  
\text{and}\  
|X| E_j = \sum_{i\in \mathcal{I}}Q_j(i)A_i,
\]
respectively.
Note that $P_i(j)$ and $Q_j(i)$ are complex numbers in general,
and if $\mathfrak{X}$ is symmetric, then $P_i(j)$ and $Q_j(i)$ are real numbers.
It is known that $k_i:= P_{i}(j_0)$ and $m_j:=Q_{j}(i_0)$ are positive integers,
and $k_i$ and $m_j$ are called the \emph{valencies} and the \emph{multiplicities} of $\mathfrak{X}$, respectively.
Also one can check easily that $P_{i_0}(j)=1$ and $Q_{j_0}(i)=1$ hold.
Then the following formula of the Krein number holds:
\begin{equation}
    \label{eq:Krein}
    q_{ij}^k=\frac{m_i m_j}{|X|}\sum_{l\in \mathcal{I}} \frac{1}{k_l^2} P_l (i)P_l (j)\overline{P_l (k)}
\end{equation}
for $i,j,k\in \mathcal{J}$.
Here $\overline{P_l (k)}$ denotes the complex conjugate of $P_l (k)$.
By $P_l (i)/k_l=\overline{Q_i(l)}/m_i$, the above formula \eqref{eq:Krein} can be rewritten as
\begin{equation}
    \label{eq:Krein2}
    q_{ij}^k=\frac{1}{|X|m_k}\sum_{l\in \mathcal{I}} k_l \overline{Q_i (l)}\overline{Q_j (l)}Q_k (l).
\end{equation}
By $q_{i,j_0}^j=\delta_{ij}$, where $\delta_{ij}$ is the Kronecker delta,
we have
\begin{equation}
    \label{eq:Krein3}
\sum_{l\in \mathcal{I}} k_l \overline{Q_i (l)}Q_j (l)=|X|m_i \delta_{ij}.
\end{equation}
Note that \eqref{eq:Krein3} is known as the orthogonality relation of the second eigenmatrix $Q$.

A symmetric association scheme $\mathfrak{X}=(X,\{R_i\}_{i\in \mathcal{I}})$ of class $d$
is called \emph{$P$-polynomial}
if it satisfies the following conditions:
$\mathcal{I}=\{0,1,\ldots ,d\}$
and there exists a univariate polynomial $v_i$ of degree $i$ such that $A_i=v_i(A_1)$
for each $i\in \{0,1,\ldots ,d\}$. 
Similarly, a symmetric association scheme $\mathfrak{X}=(X,\{R_i\}_{i\in \mathcal{I}})$ of class $d$
is called \emph{$Q$-polynomial}
if it satisfies the following conditions:
$\mathcal{J}=\{0,1,\ldots ,d\}$
and there exists a univariate polynomial $v^\ast_j$ of degree $j$ such that $|X|E_j=v^\ast _j(|X| E_1)$
(under the Hadamard product)
for each $j\in \{0,1,\ldots ,d\}$. 
The following condition is well known as an equivalent condition of the property of $P$-polynomial: for $i\in \{0,1,\ldots,d\}$, the three-term recurrence formula 
\[
    A_1 A_i = p^{i-1}_{1i} A_{i-1} +p^{i}_{1i} A_i +p^{i+1}_{1i} A_{i+1}
\]
holds.
Note that $p^{-1}_{10}A_{-1}$ and $p^{d+1}_{1d}A_{d+1}$ are regarded as zero.
Similarly, an equivalent condition of the property of $Q$-polynomial
is the following: for $i\in \{0,1,\ldots ,d\}$, the three-term recurrence formula 
\[
    (|X| E_1) \circ (|X| E_i) = q^{i-1}_{1i}|X| E_{i-1} +q^{i}_{1i} |X| E_i +q^{i+1}_{1i} |X| E_{i+1}
\]
holds. 
Note that $q^{-1}_{10}|X|E_{-1}$ and $q^{d+1}_{1d}|X|E_{d+1}$ are regarded as zero.

\subsection{Multivariate polynomial association schemes}
\label{sec:Def}

To introduce the definition of multivariate polynomial association schemes,
we first explain the fundamentals of monomial orders.
For further details, please refer to Cox-Little-O'Shea~\cite{Cox}.
In this paper, we will use the following notation:
\[\N^\ell:=\{(n_1,n_2,\ldots ,n_\ell) \mid \text{$n_i$ are nonnegative integers}\}.\] 
Here, $o := (0,0,\ldots ,0)\in \N^\ell$ represents the $\ell$-tuple of zeros. 
For $i=1,2,\ldots ,\ell$, $\epsilon_i\in \N^\ell$ denotes the $\ell$-tuple in which the $i$-th entry is $1$ and the remaining entries are $0$. 
For any $\alpha=(n_1,n_2, \ldots ,n_\ell), \beta=(m_1,m_2, \ldots ,m_\ell)\in \N^\ell$, we define $\alpha+\beta$ as $(n_1+m_1, n_2+m_2, \ldots , n_\ell +m_\ell)$, 
and when $n_i\ge m_i$ ($i=1,2,\ldots , \ell$), we define $\alpha-\beta$ as $(n_1-m_1, n_2-m_2, \ldots , n_\ell -m_\ell)$. 
For $\alpha=(n_1,n_2, \ldots ,n_\ell)\in \N^\ell$, we denote $\sum^\ell_{i=1}n_i$ by $|\alpha|$.

\begin{dfe}
    \label{df:monomialorder}
    A \emph{monomial order} $\le$ on $\N^\ell$
    is a relation on the set of $\N^\ell$ satisfying:
    \begin{enumerate}[label=$(\roman*)$]
        \item $\le$ is a total order;
        \item for $\alpha,\beta,\gamma\in \N^\ell$, if $\alpha\le \beta$, then $\alpha + \gamma \le \beta + \gamma$;
        \item $\le$ is a well-ordering, i.e., any nonempty subset of $\N^\ell$
        has a minimum element under $\le$.
    \end{enumerate}
\end{dfe}

Here we give two typical examples of monomial orders.
Let $\alpha=(n_1,n_2, \ldots ,n_\ell), \beta=(m_1,m_2, \ldots ,m_\ell)\in \N^\ell$. 
We define $\alpha \le_{\mathrm{lex}} \beta$ if the leftmost nonzero entry of $\alpha-\beta\in \Z^\ell$ is negative. 
This  $\le_{\mathrm{lex}}$ becomes a monomial order, and is called the \emph{lexicographic} (or \emph{lex}) order. 
Furthermore, we define $\alpha \le_{\mathrm{grlex}} \beta$ if $|\alpha|<|\beta|$ or both $|\alpha|=|\beta|$ and $\alpha \le_{\mathrm{lex}} \beta$ hold. 
This $\le_{\mathrm{grlex}}$ becomes also a monomial order, and
is known as the \emph{graded lexicographic} (or \emph{grlex}) order.

For $\alpha=(n_1,n_2,\ldots , n_\ell)\in \N^\ell$
and $\bm{x}=(x_1,x_2,\ldots , x_\ell)$,
we write the monomial $x_1^{n_1}x_2^{n_2}\cdots x_\ell^{n_\ell}$
by $\bm{x}^\alpha$.
Then $\alpha$ is called the \emph{multidegree} of $\bm{x}^\alpha$.

\begin{dfe}[\cite{BKZZ}]
    \label{df:abPpoly}
    Let $\mathcal{D}\subset \N^\ell$
    having $\epsilon_1,\epsilon_2,\ldots ,\epsilon_\ell$
    and $\le$ be a monomial order on $\N^\ell$.
    A commutative association scheme $\mathfrak{X}=(X,\mathcal{R})$ is called \emph{$\ell$-variate $P$-polynomial}
    on the domain $\mathcal{D}$ with respect to $\le$
    if the following three conditions are satisfied:
    \begin{enumerate}[label=$(\roman*)$]
        \item if $(n_1,n_2,\ldots ,n_\ell)\in \mathcal{D}$
        and $0\le m_i \leq n_i$ for $i=1,2,\ldots ,\ell$,
        then $(m_1,m_2,\ldots ,m_\ell)\in \mathcal{D}$;
        \item there exists a relabeling of the adjacency matrices of $\mathfrak{X}$:
        \[
        \{A_i\}_{i\in \mathcal{I}} = \{A_\alpha \}_{\alpha\in \mathcal{D}},   
        \]
        such that, for $\alpha\in \mathcal{D}$,
        \begin{equation}
            \label{eq:A=vij}
            A_\alpha=v_\alpha(A_{\epsilon_1},A_{\epsilon_2},\ldots ,A_{\epsilon_\ell}),    
        \end{equation}
        where $v_\alpha(\bm{x})$ is an $\ell$-variate polynomial
        of multidegree $\alpha$ with respect to $\le$
        and all monomials $\bm{x}^\beta$ in $v_\alpha(\bm{x})$ satisfy $\beta \in \mathcal{D}$;
        \item for $i=1,2,\ldots ,\ell$ and $\alpha=(n_1,n_2,\ldots,n_\ell)\in \mathcal{D}$,
        the product
        $A_{\epsilon_i}\cdot A_{\epsilon_1}^{n_1}A_{\epsilon_2}^{n_2}\cdots A_{\epsilon_\ell}^{n_\ell}$
        is a linear combination of
        \[
         \{A_{\epsilon_1}^{m_1}A_{\epsilon_2}^{m_2}\cdots A_{\epsilon_\ell}^{m_\ell} \mid \beta =(m_1,m_2,\ldots ,m_\ell)\in \mathcal{D},\  \beta \le  \alpha +\epsilon_i\}.   
        \]
    \end{enumerate}

\end{dfe}

Hereafter, we use notation $\bm{A}$ to denote $(A_{\epsilon_1},A_{\epsilon_2},\ldots ,A_{\epsilon_\ell})$.
Also, for $\alpha=(n_1,n_2,\ldots , n_\ell)\in \N^\ell$,
we write the monomial $A_{\epsilon_1}^{n_1} A_{\epsilon_2}^{n_2} \cdots A_{\epsilon_\ell}^{n_\ell}$
by $\bm{A}^\alpha$.

\begin{prop}[cf.~Proposition~2.6 of \cite{bi}]
\label{prop:P-TFAE}
Let $\mathcal{D}\subset \N^\ell$ having $\epsilon_1,\epsilon_2,\ldots ,\epsilon_\ell$ and
$\mathfrak{X} = (X,\{A_\alpha\}_{\alpha\in \mathcal{D}})$
be a commutative association scheme.
Then the statements (i) and (ii) are equivalent:
\begin{enumerate}[label=$(\roman*)$]
\item $\mathfrak{X}$ is
an $\ell$-variate $P$-polynomial association scheme 
on $\mathcal{D}$ with respect to a monomial order $\le$;
\item
the condition (i) of Definition~\ref{df:abPpoly} holds for $\mathcal{D}$
and the intersection numbers satisfy,
for each $i=1,2,\ldots ,\ell$ and each $\alpha \in \mathcal{D}$,
$p^\beta_{\epsilon_i, \alpha} \neq 0$ for $\beta\in \mathcal{D}$ implies $\beta \le \alpha+\epsilon_i$. 
Moreover, if $\alpha+\epsilon_i\in \mathcal{D}$, then $p^{\alpha+\epsilon_i}_{\epsilon_i, \alpha} \neq 0$ holds.
\end{enumerate}
\end{prop}

The multivariate $Q$-polynomial association scheme is defined by replacing the adjacency matrices with the primitive idempotents in the definition of the multivariate $P$-polynomial association scheme as Definition~\ref{df:abQpoly}. 
An equivalent condition for the multivariate $Q$-polynomial association scheme is also given by Proposition~\ref{prop:Q-TFAE} in the same way as Proposition~\ref{prop:P-TFAE}.

\begin{dfe}[\cite{BKZZ}]
    \label{df:abQpoly}
    Let $\mathcal{D}^\ast \subset \N^\ell$
    having $\epsilon_1,\epsilon_2,\ldots ,\epsilon_\ell$
    and $\le$ be a monomial order on $\N^\ell$.
    A commutative association scheme $\mathfrak{X}=(X,\mathcal{R})$
    with the primitive idempotents $\{E_j\}_{j\in \mathcal{J}}$ is called \emph{$\ell$-variate $Q$-polynomial}
    on the domain $\mathcal{D}^\ast$ with respect to $\le$
    if the following three conditions are satisfied:
    \begin{enumerate}[label=$(\roman*)$]
        \item if $(n_1,n_2,\ldots ,n_\ell)\in \mathcal{D}^\ast$
        and $0\le m_i \leq n_i$ for $i=1,2,\ldots ,\ell$,
        then $(m_1,m_2,\ldots ,m_\ell)\in \mathcal{D}^\ast$;
        \item there exists a relabeling of the adjacency matrices:
        \[
        \{E_j\}_{j\in \mathcal{J}} = \{E_\alpha \}_{\alpha\in \mathcal{D}^\ast}, 
        \]
        such that, for $\alpha\in \mathcal{D}^\ast$,
        \[
            |X| E_\alpha=v^\ast_\alpha(|X| E_{\epsilon_1},|X| E_{\epsilon_2},\ldots ,|X| E_{\epsilon_\ell})
            \ \text{(under the Hadamard product),}
        \]
        where $v^\ast_\alpha(\bm{x})$ is an $\ell$-variate polynomial
        of multidegree $\alpha$ with respect to $\le$
        and all monomials $\bm{x}^\beta$ in $v^\ast_\alpha(\bm{x})$ satisfy $\beta \in \mathcal{D}^\ast$;
        \item for $i=1,2,\ldots ,\ell$ and $\alpha=(n_1,n_2,\ldots,n_\ell)\in \mathcal{D}^\ast$,
        the product
        $E_{\epsilon_i}\circ E_{\epsilon_1}^{\circ n_1}\circ E_{\epsilon_2}^{\circ n_2}\circ \cdots  \circ E_{\epsilon_\ell}^{\circ n_\ell}$
        is a linear combination of
        \[
            \{E_{\epsilon_1}^{\circ m_1}\circ E_{\epsilon_2}^{\circ m_2}\circ \cdots  \circ E_{\epsilon_\ell}^{\circ m_\ell} \mid \beta=(m_1,m_2,\ldots ,m_\ell)\in \mathcal{D}^\ast,\  \beta\le  \alpha+\epsilon_i\}.   
        \]
    \end{enumerate}

\end{dfe}

\begin{prop}
    \label{prop:Q-TFAE}
    Let $\mathcal{D}^\ast\subset \N^\ell$  having $\epsilon_1,\epsilon_2,\ldots ,\epsilon_\ell$ and
    $\mathfrak{X}$
    be a commutative association scheme
    with the primitive idempotents $\{E_\alpha\}_{\alpha\in \mathcal{D}^\ast}$
    indexed by $\mathcal{D}^\ast$.
    The statements (i) and (ii) are equivalent:
    \begin{enumerate}[label=$(\roman*)$]
        \item $\mathfrak{X}$ is
        an $\ell$-variate $Q$-polynomial association scheme 
        on $\mathcal{D}^\ast$ with respect to $\le$;
        \item
        the condition (i) of Definition~\ref{df:abQpoly} holds for $\mathcal{D}^\ast$
        and the Krein numbers satisfy,
        for each $i=1,2,\ldots ,\ell$ and each $\alpha \in \mathcal{D}^\ast$,
        $q^\beta_{\epsilon_i, \alpha} \neq 0$ for $\beta\in \mathcal{D}^\ast$ implies $\beta \le \alpha+\epsilon_i$. 
        Moreover, if $\alpha+\epsilon_i\in \mathcal{D}^\ast$, then $q^{\alpha+\epsilon_i}_{\epsilon_i, \alpha} \neq 0$ holds.
    \end{enumerate}
\end{prop}

\section{Nonbinary Johnson schemes}
\label{sec:nonbinary}

In this section, we describe the nonbinary Johnson scheme,
which is a generalization of the Hamming schemes and the Johnson schemes,
and in this paper, we require various properties of these schemes. 
Therefore, 
in Subsection~\ref{sec:Hamming_Johnson}, 
we first describe the properties of the Hamming schemes and the Johnson schemes.
In Subsection~\ref{sec:nonbinaryAS}, we describe the nonbinary Johnson schemes.

\subsection{Hamming schemes and Johnson schemes}
\label{sec:Hamming_Johnson}
Let $H(n,q)$ be the \emph{Hamming scheme}, i.e., the set of all $q$-ary $n$-tuples with the adjacency relation $R_i$ defined by the Hamming distance $i$.
Then the class of $H(n,q)$ is $n$.
For details of the Hamming scheme, see \cite{BBIT2021,BI1984,BCN1989}.
Let $k^{H(n,q)}_i$ and $m^{H(n,q)}_i$ be the valency and the multiplicity of $H(n,q)$, respectively.
These values are known as
\begin{equation}
    \label{eq:Hamming_k_m}
    k^{H(n,q)}_i=m^{H(n,q)}_i=(q-1)^i\binom{n}{i}.
\end{equation}
Also, let $p^k_{ij}(H(n,q))$ and $q^k_{ij}(H(n,q))$
be the intersection number and the Krein number of $H(n,q)$, respectively.
It is well known that Hamming schemes are $P$- and $Q$-polynomial association schemes.
We give a formula for these numbers for later use:
\begin{equation}
    \label{eq:Hamming_p_q}
    p^\alphaalt_{1i}(H(k-y,r-1))=q^\alphaalt_{1i}(H(k-y,r-1))=
    \begin{cases}
        (k-y-i+1)(r-2) & \text{if $\alphaalt=i-1$;}\\
        i(r-3) & \text{if $\alphaalt=i$;}\\
        i+1 & \text{if $\alphaalt=i+1$;}\\
        0 & \text{otherwise.}
    \end{cases}
\end{equation}
Moreover, let $P^{H(n,q)}_i(j)$ and $Q^{H(n,q)}_i(j)$ be the $(i,j)$-entry of the first and the second eigenmatrix of $H(n,q)$, respectively.
Then we have
\[
    P^{H(n,q)}_i(j)=Q^{H(n,q)}_i(j)=K_i(n,q;j),
\]
where $K_i(n,q;j)$ is the \emph{Krawtchouk polynomial} defined by
\[
    K_i(n, q;j) = \sum_{u = 0}^i (-1)^{u} (q-1)^{i-u} \binom{j}{u} \binom{n-j}{i-u}.
\]

Let $J(n,k)$ be the \emph{Johnson scheme}, i.e., the set of all $k$-subsets of $\{1,2,\ldots,n\}$ with the adjacency relation $R_i$ defined by the symmetric difference of cardinality $i$.
Then the class of $J(n,k)$ is $k\wedge (n-k)$, where $k\wedge (n-k)$ is the minimum of $k$ and $n-k$.
For details of the Johnson scheme, see \cite{BBIT2021,BI1984,BCN1989}.
Let $k^{J(n,k)}_i$ and $m^{J(n,k)}_i$ be the valency and the multiplicity of $J(n,k)$, respectively.
These values are known as
\begin{align}
    k^{J(n,k)}_i&=\binom{k}{i}\binom{n-k}{i},\label{eq:Johnson_k}\\
    m^{J(n,k)}_i&=\binom{n}{i}-\binom{n}{i-1}=\frac{n-2i+1}{n-i+1}\binom{n}{i}.\label{eq:Johnson_m}
\end{align}
It is well known that Johnson schemes are $P$- and $Q$-polynomial association schemes.
Let $q^\gamma_{\alphaalt\betaalt}(J(n,k))$ be the Krein number of the Johnson scheme $J(n,k)$.
We give a formula for these numbers for later use:
\begin{equation}
    \label{eq:Johnson_q}
    q^\alphaalt_{1j}(J(n-i,k-i))=
    \begin{cases}
        \frac{(n-i)(n-i-1)(k-i-j+1)(n-i-j+2)(n-k-j+1)}{(k-i)(n-k)(n-i-2j+2)(n-i-2j+3)}
 & \text{if $\alphaalt=j-1$;}\\
        a^\ast_{j} & \text{if $\alphaalt=j$;}\\
        \frac{(n-i)(n-i-1)(j+1)(k-i-j)(n-k-j)}{(k-i)(n-k)(n-i-2j)(n-i-2j-1)} & \text{if $\alphaalt=j+1$;}\\
        0 & \text{otherwise,}
    \end{cases}
\end{equation}
where
\[
a^\ast_{j}=
\textstyle{
(n-i-1)
-\frac{(n-i)(n-i-1)}{(k-i)(n-k)(n-i-2j+1)}\left(
\frac{(k-i-j)(n-i-j+1)(n-k-j)}{(n-i-2j)}
+\frac{j(k-i-j+1)(n-k-j+1)}{(n-i-2j+2)}
\right).}    
\]
Moreover, let $P^{J(n,k)}_i(j)$ and $Q^{J(n,k)}_i(j)$ be the $(i,j)$-entry of the first and the second eigenmatrix of $J(n,k)$, respectively.
Then we have
\[
    P^{J(n,k)}_i(j)=E_i(n,k;j)
    ,\quad Q^{J(n,k)}_i(j)=H_i(n,k;j),
\]
where $E_i(n,k;j)$ and $H_i(n,k;j)$ are the \emph{Eberlein polynomial} (\emph{dual Hahn}) and the \emph{Hahn polynomial} defined by
\begin{align*}
    E_i(n, k;j) &= \sum_{u = 0}^i (-1)^{u} \binom{j}{u} \binom{k-j}{i-u} \binom{n-k-j}{i-u}, \\
    H_i(n, k;j) &= \frac{\binom{n}{i} - \binom{n}{i-1}}{\binom{k}{j} \binom{n-k}{j}} E_j(n, k;i). 
\end{align*}

For the Hahn polynomials, we have the following recurrence relation.
This relation is due to \cite{biQ}.
\begin{lem}
\label{lem:Hahn} Let $p,N$ be integers such that $0<p<N$.
For $0\le x \le (p-1)\wedge (N-p)$
and $0\le r \le p\wedge (N-p)$, we have
\begin{equation}
\label{eq:Crampe}
H_r(N,p;x)=\tfrac{N}{p} 
\left( \tfrac{p-r}{N-2r}H_r(N-1,p-1;x)+\tfrac{N-p-r+1}{N-2r+2}
H_{r-1}(N-1,p-1;x) \right).
\end{equation}
\end{lem}
Although \eqref{eq:Crampe} seems not to hold for some values of $r$,
it can be justified by interpreting it as follows.
For $r=0$, we regard $H_{-1}(N-1,p-1;x)$ as zero,
i.e., \eqref{eq:Crampe} is written by
\[
    H_0(N,p;x) = H_0(N-1,p-1;x).
\]
In fact, both $H_0(N,p;x)$ and $H_0(N-1,p-1;x)$ become 1, hence the above equation is correct. 
For $r=p$, we regard $H_p(N-1,p-1;x)$ as zero,
i.e., \eqref{eq:Crampe} is written by
\[
    H_p(N,p;x)=\frac{N(N-2p+1)}{p(N-2p+2)} H_{p-1}(N-1,p-1;x).
\]
In fact, by $H_p(N,p;x)=(-1)^x \frac{N-2p+1}{N-p+1} \frac{\binom{N}{p}}{\binom{N-p}{x}}$, the above equation is justified.

\subsection{Nonbinary Johnson schemes}
\label{sec:nonbinaryAS}

Let $n$ and $r$ be positive integers such that $r > 1$ and let $k$ be a natural number such that $0 \leq k \leq n$. 
Let $K = \set{0,1,\ldots, r-1}$ be a set of cardinality $r$. 
For a vector $\bm{x} = (x_1, x_2, \ldots, x_n) \in K^n$, its weight $w(\bm{x})$ is defined by the number of non-zero entries, namely
$w(\bm{x}) = \card{\set{i \mid x_i \neq 0 }}$.
We consider the following set $S = \set{\bm{x} \in K^n \mid w(\bm{x}) = k}$.
Note that the cardinality of $S$ is given by
\begin{equation}
    \label{eq:nonbinaryJohnson_card}
    |S|= (r-1)^k \binom{n}{k}.
\end{equation}
Let $\mathcal{D}=\set{(i,j)\mid i+j\le k,\  0\le j\le k\wedge (n-k)}$.
For two vectors $\bm{x} = (x_1, x_2, \ldots, x_n)$ and $\bm{y} = (y_1, y_2, \ldots, y_n)$ in $K^n$, we consider the number of equal non-zero entries $e(\bm{x},\bm{y})$ and the number of common non-zero entries $c(\bm{x},\bm{y})$, namely
\[
    e(\bm{x},\bm{y}) = \card{\set{i \mid x_i = y_i \neq 0}},\ 
    c(\bm{x},\bm{y}) = \card{\set{i \mid x_i \neq 0, y_i \neq 0}}. 
\]
The map $\mathcal{R}\colon S\times S \to \mathcal{D}$
is defined by $\mathcal{R}(\bm{x},\bm{y})=(i,j)$, where
\[
    e(\bm{x}, \bm{y}) = k-i-j\ \text{and}\  c(\bm{x},\bm{y}) = k-j.
\]
Then the pair $(S,\mathcal{R})$ becomes a symmetric association scheme, which we call the \emph{nonbinary Johnson scheme}. 
We denote it by $J_r(n,k)$.
Note that if $r=2$, then $J_r(n,k)$ is the Johnson scheme $J(n,k)$
and if $n=k$, then $J_r(n,k)$ is the Hamming scheme $H(n,r)$.
Therefore, we focus on the case where $r > 2$ and $n> k$. 

It is known that the indices of the primitive idempotents of $J_r(n,k)$ are also labeled by $\mathcal{D}$.
The explicit expression of eigenvalues of $J_r(n,k)$ are given by Tarnanen-Aaltonen-Goethals~\cite{TAG1985}.
For $(i,j), (x,y)\in \mathcal{D}$, we have
\begin{align}
    P_{ij}(xy) = (r-1)^j K_i(k-j, r-1;x) E_j(n-x, k-x;y).
\end{align}
Moreover, the valencies of $J_r(n,k)$ are given by
\begin{equation}
    \label{eq:nonbinaryJohnson_k}
    k_{ij}=(r-1)^j k^{H(k-j,r-1)}_i k^{J(n,k)}_j.
\end{equation}
Similarly, 
for $(i,j),(x,y)\in \mathcal{D}$,
the entries of the second eigenmatrix of $J_r(n,k)$ are given by
\begin{align}
    \label{eq:nonbinaryJohnson_q}
    Q_{ij}(xy) = \frac{\binom{n}{i}}{\binom{k}{i}} K_i(k-y, r-1;x) H_j(n-i,k-i;y).
\end{align}
Moreover, the multiplicities of $J_r(n,k)$ are given by
\begin{equation}
    \label{eq:nonbinaryJohnson_m}
    m_{ij}=\frac{\binom{n}{i}}{\binom{k}{i}} m^{H(k,r-1)}_i m^{J(n-i,k-i)}_j.
\end{equation}

\section{Proof of Theorem~\ref{thm:main1}}
\label{sec:proofThm1}

In order to prove Theorem~\ref{thm:main1},
it is sufficient to prove the following proposition from Proposition~\ref{prop:Q-TFAE}.
\begin{prop}
\label{prop:nonbinaryQ}
The Krein numbers of
$J_r(n,k)$
satisfy the following:
\begin{enumerate}[label=$(\roman*)$]
\item $q^{\alphaalt \betaalt}_{10,ij} \neq 0$ implies
\[
(\alphaalt,\betaalt) \in \set{(i+1,j),(i+1,j-1),(i,j),(i-1,j+1),(i-1,j)}
\cap \mathcal{D},
\]
i.e., $(\alphaalt,\betaalt)\le_{\mathrm{grlex}} (i,j)+(1,0)$.
The exact values of $q^{\alphaalt \betaalt}_{10,ij}$ are given by
\eqref{eq:nonbinary_Krein_1},
\eqref{eq:nonbinary_Krein_2},
\eqref{eq:nonbinary_Krein_3}
and
\eqref{eq:nonbinary_Krein_4}.
Moreover, if $(i+1,j)=(i,j)+(1,0)\in \mathcal{D}$,
then $q^{i+1,j}_{10,ij} \neq 0$ holds.

\item $q^{\alphaalt \betaalt}_{01,ij} \neq 0$ implies
\[
(\alphaalt,\betaalt) \in \set{(i,j+1),(i,j),(i,j-1)}
\cap \mathcal{D},
\]
i.e., $(\alphaalt,\betaalt)\le_{\mathrm{grlex}} (i,j)+(0,1)$.
The exact values of $q^{\alphaalt \betaalt}_{01,ij}$ are given by
\eqref{eq:nonbinary_Krein_5},
\eqref{eq:nonbinary_Krein_6}
and
\eqref{eq:nonbinary_Krein_7}.
Moreover, if $(i,j+1)=(i,j)+(0,1)\in \mathcal{D}$,
then $q^{i,j+1}_{01,ij} \neq 0$ holds.
\end{enumerate}
\end{prop}

\begin{Rem}
    Proposition~\ref{prop:nonbinaryQ} was essentially already shown by \cite{biQ}.
    However, the proof is given again using an approach different from \cite{biQ} in the following.
    This calculation will help in understanding Section~\ref{sec:proofThm2}.
\end{Rem}

\subsection{Krein numbers $q^{\alphaalt \betaalt}_{10,ij}$}
\label{sec:nonbinaryQ_10}
To prove (i) of Proposition~\ref{prop:nonbinaryQ},
we compute the Krein numbers $q^{\alphaalt \betaalt}_{10,ij}$ of the nonbinary Johnson scheme
using the formula \eqref{eq:Krein2}.
Note that $J_r(n,k)$ is symmetric, so there is no need to consider complex conjugates.
By \eqref{eq:nonbinaryJohnson_card},
\eqref{eq:nonbinaryJohnson_k},
\eqref{eq:nonbinaryJohnson_q} and
\eqref{eq:nonbinaryJohnson_m},
we have
\begin{align*}
q^{\alphaalt \betaalt}_{10,ij}
=&
\frac{1}{|S| m_{\alphaalt \betaalt}}
\sum_{(x,y)\in \mathcal{D}}
k_{xy} Q_{10}(xy)Q_{ij}(xy)Q_{\alphaalt \betaalt}(xy)\\
=&
\frac{n\binom{n}{i}}{k \binom{k}{i} \binom{n}{k}(r-1)^k m^{H(k,r-1)}_\alphaalt m^{J(n-\alphaalt,k-\alphaalt)}_\betaalt}\\
&\times \sum_{y=0}^{k \wedge (n-k)}(r-1)^y k^{J(n,k)}_y
H_j(n-i,k-i;y)H_\betaalt(n-\alphaalt,k-\alphaalt;y)\\
&\times \sum_{x=0}^{k-y}
k^{H(k-y,r-1)}_x
K_1(k-y,r-1;x)K_i(k-y,r-1;x)K_\alphaalt(k-y,r-1;x).
\end{align*}
Applying \eqref{eq:Krein2} to the Hamming scheme $H(k-y,r-1)$,
we have
\begin{align}
    &\sum_{x=0}^{k-y}
    k^{H(k-y,r-1)}_x
    K_1(k-y,r-1;x)K_i(k-y,r-1;x)K_\alphaalt(k-y,r-1;x) \notag\\
    &=
    (r-1)^{k-y}m^{H(k-y,r-1)}_\alphaalt 
    q^\alphaalt_{1i}(H(k-y,r-1)). \label{eq:H(k-y,r-1)}
\end{align}
Using 
\eqref{eq:Hamming_k_m}, \eqref{eq:Johnson_k},
\eqref{eq:H(k-y,r-1)} and
$\binom{n}{i}\binom{n-i}{k-i}=\binom{n}{k}\binom{k}{i}$,
we have
\begin{align}
q^{\alphaalt \betaalt}_{10,ij}
=&
\frac{n}{k \binom{n-i}{k-i} \binom{k}{\alphaalt} m^{J(n-\alphaalt,k-\alphaalt)}_\betaalt}
\sum_{y=0}^{k \wedge (n-k)} \binom{k}{y}\binom{n-k}{y}
H_j(n-i,k-i;y)H_\betaalt(n-\alphaalt,k-\alphaalt;y)\notag \\
&\times \binom{k-y}{\alphaalt}
q^\alphaalt_{1i}(H(k-y,r-1)).
\label{eq:Krein_alpha_beta}
\end{align}
According to \eqref{eq:Hamming_p_q},
$\alphaalt$ must be $i-1,i$ or $i+1$.
In the below argument, we will use
\begin{equation}
    \label{eq:binom_rel}
    \binom{k}{y}\binom{k-y}{i}=\binom{k}{i}\binom{k-i}{y}.
\end{equation}

When the case $\alphaalt=i$ in \eqref{eq:Krein_alpha_beta},
by \eqref{eq:Johnson_k} and \eqref{eq:binom_rel},
we have
    \[
    q^{i \betaalt}_{10,ij}
    =
    \frac{n i (r-3)}{k \binom{n-i}{k-i} m^{J(n-i,k-i)}_\betaalt}
    \sum_{y=0}^{k \wedge (n-k)} k^{J(n-i,k-i)}_y
    H_j(n-i,k-i;y)H_\betaalt(n-i,k-i;y).
    \]
Since $H_j(n-i,k-i;y)=0$ if $y > (k-i)\wedge ((n-i)-(k-i))$,
the range of $y$ in the above summation will be $0\le y \le (k-i)\wedge (n- k)$.
The orthogonal relation \eqref{eq:Krein3} for $J(n-i,k-i)$ implies
\[
q^{i \betaalt}_{10,ij}=
    \frac{n i (r-3)}{k \binom{n-i}{k-i} m^{J(n-i,k-i)}_\betaalt}
    |J(n-i,k-i)| m^{J(n-i,k-i)}_j \delta_{j\betaalt}
\]
and $\betaalt$ must be $j$.
Since $|J(n-i,k-i)| =\binom{n-i}{k-i}$,
we have
\begin{equation}
    \label{eq:nonbinary_Krein_1}
    q^{i j}_{10,ij}=\frac{n i (r-3)}{k}.
\end{equation}

When the case $\alphaalt=i+1$ in \eqref{eq:Krein_alpha_beta},
by \eqref{eq:binom_rel} and \eqref{eq:Johnson_k}
we have
\[
q^{i+1, \betaalt}_{10,ij}
=
\frac{n(i+1)}{k \binom{n-i}{k-i} m^{J(n-i-1,k-i-1)}_\betaalt}
\sum_{y=0}^{k \wedge (n-k)}
k^{J(n-i-1,k-i-1)}_y
H_j(n-i,k-i;y)H_\betaalt(n-i-1,k-i-1;y).
\]
Since $H_\betaalt(n-i-1,k-i-1;y)=0$ if $y > (k-i-1)\wedge ((n-i-1)-(k-i-1))$,
the range of $y$ in the above summation will be $0\le y \le (k-i-1)\wedge (n- k)$.
From Lemma~\ref{lem:Hahn},
\[
\textstyle{
H_j(n-i,k-i;y)
=
\frac{n-i}{k-i}  
\left(
\frac{k-i-j}{n-i-2j} H_j(n-i-1,k-i-1;y)
+
\frac{n-k-j+1}{n-i-2j+2}H_{j-1}(n-i-1,k-i-1;y)
\right)
}
\]
holds.
Then the orthogonal relation \eqref{eq:Krein3}  for $J(n-i-1,k-i-1)$ implies
\begin{align*}
q^{i+1, \betaalt}_{10,ij}
=&
\frac{n(i+1)(n-i)}{k (k-i) \binom{n-i}{k-i} m^{J(n-i-1,k-i-1)}_\betaalt}
|J(n-i-1,k-i-1)|m^{J(n-i-1,k-i-1)}_\betaalt\\
&\times \left(
\frac{k-i-j}{n-i-2j} \delta_{\betaalt ,j}
+
\frac{n-k-j+1}{n-i-2j+2}\delta_{\betaalt ,j-1}
\right)\\
=&
\frac{n(i+1)}{k}
\left(
\frac{k-i-j}{n-i-2j} \delta_{\betaalt ,j}
+
\frac{n-k-j+1}{n-i-2j+2}\delta_{\betaalt ,j-1}
\right).
\end{align*}
In the last line, we used
\[
    |J(n-i-1,k-i-1)|=\binom{n-i-1}{k-i-1}=\binom{n-i}{k-i}\frac{k-i}{n-i}.
\]
Thus, $\betaalt$ must be $j$ or $j-1$, and we have
\begin{equation}
    \label{eq:nonbinary_Krein_2}
    q^{i+1,j}_{10,ij} = \frac{n(i+1)(k-i-j)}{k(n-i-2j)},
\quad
q^{i+1,j-1}_{10,ij} = \frac{n(i+1)(n-k-j+1)}{k(n-i-2j+2
)}.
\end{equation}
Moreover, if $(i+1,j)\in \mathcal{D}$,
then $(i+1)+j\le k$, i.e., $k-i-j\ge 1$ and $j\le n-k$, i.e., $n-i-2j\ge 1$.
This implies $q^{i+1,j}_{10,ij} \neq 0$.

When the case $\alphaalt=i-1$ in \eqref{eq:Krein_alpha_beta},
by
\[
\binom{k}{i-1}=\frac{i}{k-i+1}\binom{k}{i}
\quad  \text{and}\quad  
\binom{k-y}{i-1}= \frac{i}{k-y-i+1}\binom{k-y}{i},
\]
\eqref{eq:binom_rel} and \eqref{eq:Johnson_k},
we have
\[
q^{i-1, \betaalt}_{10,ij}
=
\frac{n(r-2)(k-i+1)}{k \binom{n-i}{k-i} m^{J(n-i+1,k-i+1)}_\betaalt}
\sum_{y=0}^{k \wedge (n-k)} k^{J(n-i,k-i)}_y
H_j(n-i,k-i;y)H_\betaalt(n-i+1,k-i+1;y).
\]
Since $H_j(n-i,k-i;y)=0$ if $y > (k-i)\wedge ((n-i)-(k-i))$,
the range of $y$ in the above summation will be $0\le y \le (k-i)\wedge (n- k)$.
From Lemma~\ref{lem:Hahn},
\begin{align*}
    &H_\betaalt(n-i+1,k-i+1;y)\\
    &=
    \frac{n-i+1}{k-i+1}  
    \left(
    \frac{k-i+1-\betaalt}{n-i+1-2\betaalt} H_\betaalt(n-i,k-i;y)
    +
    \frac{n-k-\betaalt+1}{n-i-2\betaalt +3}H_{\betaalt-1}(n-i,k-i;y)
    \right).
\end{align*}
holds.
Then the orthogonal relation \eqref{eq:Krein3} for $J(n-i,k-i)$ implies
\[
q^{i-1, \betaalt}_{10,ij}
=
\frac{n(r-2)(n-i+1)}{k m^{J(n-i+1,k-i+1)}_\betaalt}
m^{J(n-i,k-i)}_j
\left(
\frac{k-i+1-\betaalt}{n-i+1-2\betaalt} \delta_{j,\betaalt}
+
\frac{n-k-\betaalt+1}{n-i-2\betaalt +3}\delta_{j,\betaalt-1}
\right).
\]
Thus, $\betaalt$ must be $j$ or $j+1$.
If $\betaalt=j$, then the Krein number becomes
\begin{align*}
    q^{i-1, j}_{10,ij}
    =&
    \frac{n(r-2)(n-i+1)}{k m^{J(n-i+1,k-i+1)}_j}
    m^{J(n-i,k-i)}_j
    \frac{k-i+1-j}{n-i+1-2j}.
\end{align*}
By \eqref{eq:Johnson_m}, the above multiplicities can be calculated as
\[
    m^{J(n-i+1,k-i+1)}_j
    =\frac{n-i-2j+2}{n-i-j+2}\binom{n-i+1}{j}
    =\frac{(n-i-2j+2)(n-i+1)}{(n-i-j+2)(n-i-j+1)}\binom{n-i}{j}
\]
and
\[
    m^{J(n-i,k-i)}_j=\frac{n-i-2j+1}{n-i-j+1}\binom{n-i}{j}.
\]
Thus, we have
\begin{equation}
    q^{i-1, j}_{10,ij}=
    \frac{n(r-2)(n-i-j+2)(n-i-j+1)}{k(n-i-2j+2)}.
    \label{eq:nonbinary_Krein_3}
\end{equation}
If $\betaalt=j+1$, the Krein number becomes
\begin{align*}
    q^{i-1, j+1}_{10,ij}
    =&
    \frac{n(r-2)(n-i+1)}{k m^{J(n-i+1,k-i+1)}_{j+1}}
    m^{J(n-i,k-i)}_j
    \frac{n-k-j}{n-i-2j+1}.
\end{align*}
By \eqref{eq:Johnson_m}, the multiplicity $m^{J(n-i+1,k-i+1)}_{j+1}$ can be calculated as
\[
    m^{J(n-i+1,k-i+1)}_{j+1}
    =\frac{n-i-2j}{n-i-j+1}\binom{n-i+1}{j+1}
    =\frac{(n-i-2j)(n-i+1)}{(n-i-j+1)(j+1)}\binom{n-i}{j}.
\]
Thus, we have
\begin{equation}
    q^{i-1, j+1}_{10,ij}=
    \frac{n(r-2)(j+1)(n-k-j)}{k (n-i-2j)}.
    \label{eq:nonbinary_Krein_4}
\end{equation}

\subsection{Krein numbers $q^{\alphaalt \betaalt}_{01,ij}$}
We show (ii) of Proposition~\ref{prop:nonbinaryQ} similarly as in Section~\ref{sec:nonbinaryQ_10}.
By \eqref{eq:nonbinaryJohnson_card},
\eqref{eq:nonbinaryJohnson_k},
\eqref{eq:nonbinaryJohnson_q} and
\eqref{eq:nonbinaryJohnson_m},
we have
\begin{align*}
q^{\alphaalt \betaalt}_{01,ij}
=&
\frac{1}{|S| m_{\alphaalt \betaalt}}
\sum_{y=0}^{k \wedge (n-k)}\sum_{x=0}^{k-y}
k_{xy} Q_{01}(xy)Q_{ij}(xy)Q_{\alphaalt \betaalt}(xy)\\
=&
\frac{\binom{n}{i}}{\binom{n}{k}\binom{k}{i}(r-1)^k  m^{H(k,r-1)}_\alphaalt m^{J(n-\alphaalt,k-\alphaalt)}_\betaalt}\\
&\times \sum_{y=0}^{k \wedge (n-k)}
(r-1)^y  k^{J(n,k)}_y
H_1(n,k;y)H_j(n-i,k-i;y)H_\betaalt(n-\alphaalt,k-\alphaalt;y)\\
&\times
\sum_{x=0}^{k-y}k^{H(k-y,r-1)}_x K_i(k-y,r-1;x)K_\alphaalt(k-y,r-1;x).
\end{align*}
Applying \eqref{eq:Krein3} to the Hamming scheme $H(k-y,r-1)$,
we have
\[
\sum_{x=0}^{k-y}k^{H(k-y,r-1)}_x K_i(k-y,r-1;x)K_\alphaalt(k-y,r-1;x)
=
(r-1)^{k-y}m^{H(k-y,r-1)}_\alphaalt \delta_{i,\alphaalt}.  
\]
Hence, $\alphaalt$ must be $i$ and 
by \eqref{eq:Hamming_k_m},
we have
\begin{align}
    q^{i \betaalt}_{01,ij}
=&
\frac{\binom{n}{i}}{\binom{n}{k}\binom{k}{i} ^2
m^{J(n-i,k-i)}_\betaalt}
\sum_{y=0}^{k \wedge (n-k)}
 k^{J(n,k)}_y\binom{k-y}{i} H_1(n,k;y)
H_j(n-i,k-i;y)H_\betaalt(n-i,k-i;y).
\label{eq:Krein_alpha_beta_2}
\end{align}
Since $H_j(n-i,k-i;y)=0$ if $y > (k-i)\wedge ((n-i)-(k-i))$,
the range of $y$ in the above summation will be $0\le y \le (k-i)\wedge (n- k)$.
To carry out the above summation calculations,
we replace $H_1(n,k;y)$ with another expression.
By the definition of the Hahn polynomials, we have
\begin{align}
    H_1(n,k;y)&=\frac{n-1}{k(n-k)}(k(n-k)-ny),\label{eq:H_1(y,n,k)}\\
    H_1(n-i,k-i;y)&=\frac{n-i-1}{(k-i)(n-k)}((k-i)(n-k)-(n-i)y). \label{eq:H_1(y,n-i,k-i)}
\end{align}
Eliminating $y$ from \eqref{eq:H_1(y,n,k)} and \eqref{eq:H_1(y,n-i,k-i)} yields the following equation:
\begin{equation}
    H_1(n,k;y)=
    \frac{n-1}{k}
\left(
\frac{i(n-k)}{n-i}+
\frac{n(k-i)}{(n-i-1)(n-i)}H_1(n-i,k-i;y)
\right). \label{eq:theta-ast}
\end{equation}
Thus, by 
\eqref{eq:Johnson_k},
\eqref{eq:binom_rel},
\eqref{eq:theta-ast},
\eqref{eq:Krein2} and \eqref{eq:Krein3} for $J(n-i,k-i)$,
we have
\begin{align*}
q^{i \betaalt}_{01,ij}
=&
\frac{\binom{n}{i}}{\binom{n}{k}\binom{k}{i}
m^{J(n-i,k-i)}_\betaalt}
|J(n-i,k-i)|m^{J(n-i,k-i)}_\betaalt\\
&\times \frac{n-1}{k} \left(
\frac{i(n-k)}{n-i} \delta_{j,\betaalt}+
\frac{n(k-i)}{(n-i-1)(n-i)}
q^\betaalt_{1j}(J(n-i,k-i))
\right)\\
=&
\frac{n-1}{k}\left(\frac{i(n-k)}{n-i} \delta_{j,\betaalt}+
\frac{n(k-i)}{(n-i-1)(n-i)}
q^\betaalt_{1j}(J(n-i,k-i))
\right).
\end{align*}
In the last line, we used
\[
    |J(n-i,k-i)|=\binom{n-i}{k-i}\ \text{and}\ 
    \binom{n}{i}\binom{n-i}{k-i}=\binom{n}{k}\binom{k}{i}.
\]
This implies $\betaalt$ must be $j,j+1$ or $j- 1$, and we have
\begin{align}
q^{i j}_{01,ij} 
& = \frac{n-1}{k} \left(
\frac{i(n-k)}{n-i} +\frac{n(k-i)}{(n-i-1)(n-i)}
q^j_{1j}(J(n-i,k-i))\right)\notag\\
&=
\textstyle{
n-1
-\frac{n(n-1)}{k(n-k)(n-i-2j+1)}\left(
\frac{(k-i-j)(n-i-j+1)(n-k-j)}{(n-i-2j)}
+\frac{j(k-i-j+1)(n-k-j+1)}{(n-i-2j+2)}
\right)};
\label{eq:nonbinary_Krein_5}\\
    q^{i,j-1}_{01,ij} & =
\frac{n-1}{k} \frac{n(k-i)}{(n-i-1)(n-i)}
q^{j-1}_{1j}(J(n-i,k-i))\notag\\
& =
\frac{n(n-1)(k-i-j+1)(n-i-j+2)(n-k-j+1)}{k(n-k)(n-i-2j+2)(n-i-2j+3)};
\label{eq:nonbinary_Krein_6}\\
    q^{i,j+1}_{01,ij} & =
\frac{n-1}{k} \frac{n(k-i)}{(n-i-1)(n-i)}
q^{j+1}_{1j}(J(n-i,k-i))\notag\\
&=
\frac{n(n-1)(j+1)(k-i-j)(n-k-j)}{k(n-k)(n-i-2j)(n-i-2j-1)}.
\label{eq:nonbinary_Krein_7}
\end{align}
Moreover if $(i,j+1)\in \mathcal{D}$,
then $i+(j+1)\le k$, i.e., $k-i-j\ge 1$ and $j+1\le n-k$, i.e., $n-i-2j-1\ge 1$.
This implies $q^{i j+1}_{01,ij} \neq 0$.
Note that $q^{ij}_{01,ij}$
coincides with
$\widehat{\cB}_{ij}$
in \cite{biQ}.

\begin{Rem}
$J_r(n,k)$ can be regarded as the Gelfand pair
\[(\mathfrak{S}_{r-1} \wr \mathfrak{S}_n,(\mathfrak{S}_{r-2} \wr \mathfrak{S}_k)\times (\mathfrak{S}_{r-1} \wr \mathfrak{S}_{n-k})).\]
A part of the proof of Proposition~\ref{prop:nonbinaryQ} can also be shown using the representation theory of $\mathfrak{S}_{r-1} \wr \mathfrak{S}_n$.
Specifically, the irreducible representations of $\mathfrak{S}_{r-1} \wr \mathfrak{S}_n$ are known from Stein~\cite{Stein2017},
and the irreducible representations of $\mathfrak{S}_{r-1} \wr \mathfrak{S}_n$ that appears in the permutation representation of the above Gelfand pair are determined by the result of Ceccherini-Silberstein, Scarabotti, Tolli~\cite{CST2006}.
Then, using the Littlewood-Richardson rule for $\mathfrak{S}_{r-1} \wr \mathfrak{S}_n$, 
we can investigate for which $(\alphaalt,\betaalt)$, Krein numbers $q^{\alphaalt\betaalt}_{10,ij}$ and $q^{\alphaalt\betaalt}_{01,ij}$ become zero.
However, it is so far difficult to show  $q^{i+1,j}_{10,ij}\neq 0$ and $q^{i,j+1}_{01,ij}\neq 0$ using this method.
\end{Rem}

\section{Association schemes obtained from attenuated spaces}
\label{sec:attenuated}

In this section, we describe the association schemes obtained from attenuated spaces,
which is a generalization of the bilinear forms schemes and the Grassmann schemes,
and in this paper, we require various properties of these schemes. 
Therefore, 
in Subsection~\ref{sec:bilinear_Grassmann}, 
we first describe the properties of the bilinear forms schemes and the Grassmann schemes.
In Subsection~\ref{sec:ASattenuated}, we describe the association schemes obtained from attenuated spaces.
In this section,
let $q$ be a prime power, and $n$, $m$ and $l$ be positive integers.

\subsection{Bilinear forms schemes and Grassmann schemes}
\label{sec:bilinear_Grassmann}
Let $\mathbb{F}_q$ be the finite field of size $q$,
and let $V$ and $E$ be $n$-dimensional and $l$-dimensional vector spaces
over $\mathbb{F}_q$, respectively,
and let $L(V,E)$ denote the set of all linear maps from $V$ to $E$.
Then the size of $L(V,E)$ is $q^{n l}$.
The set $L(V,E)$ together with the nonempty relations
\[R_i=\set{(f,g)\in L(V,E) \times L(V,E) \mid \rank (f-g)=i}\]
is an $n\wedge l$-class symmetric association scheme called the
\emph{bilinear forms scheme} $H_q(n,l)$.
It is well known that bilinear forms schemes are $P$- and $Q$-polynomial association schemes.
The first eigenmatrix $P=(P^{H_q(n,l)}_i(j))$ of $H_q(n,l)$ is given by the \emph{generalized Krawtchouk polynomials}
$K_i(n,l;q;j)$
(see Delsarte~\cite{Delsarte1978bfo}), namely
\begin{eqnarray*}
P^{H_q(n,l)}_i(j)&=&K_i(n,l;q;j)\\
&=&\sum^{i}_{u=0}(-1)^{i-u}q^{u l+\binom{i-u}{2}}\qbinom{n-u}{n-i}
\qbinom{n-j}{u},
\end{eqnarray*}
where 
\[
    \qbinom{n}{m}:=\prod ^{m-1}_{i=0}\frac{q^{n}-q^i}{q^{m}-q^i}
\]
is called the \emph{$q$-binomial coefficient} for $m<n$.
If either $i$ or $j$ is outside $\{0,1,\ldots ,n\wedge l\}$,
then we define $K_i(n,l;q;j)=0$.
Note that bilinear forms schemes are self-dual, i.e., 
the second eigenmatrix of $H_q(n,l)$ coincides with the first eigenmatrix of $H_q(n,l)$
(see~\cite{Delsarte1978bfo}).
Moreover, the valencies $k^{H_q(n,l)}_i$ and the multiplicities $m^{H_q(n,l)}_i$
of $H_q(n,l)$ are given as
\begin{equation}
\label{eq:bilinear_k_m}
k^{H_q(n,l)}_i=
m^{H_q(n,l)}_i=\qbinom{n}{i}\qbinom{l}{i}
\prod^{i-1}_{u=0}(q^i-q^{u}).
\end{equation}
Also, the Krein numbers $q^\alphaalt_{1i}(H_q(m-y,l))$
coincide with
the  intersection numbers $p^\alphaalt_{1i}(H_q(m-y,l))$
and these numbers
satisfy
\begin{equation}
    \label{eq:bilinear_p_q}
    q^\alphaalt_{1i}(H_q(m-y,l))=
    \begin{cases}
        q^{2i-2}(q-1)[l-i+1][m-y-i+1] & \text{if $\alphaalt=i-1$;}\\
        [i](q^{m-y}+q^{l}-q^i-q^{i-1}-1) & \text{if $\alphaalt=i$;}\\
        q^i [i+1] & \text{if $\alphaalt=i+1$;}\\
        0 & \text{otherwise,}
    \end{cases}
\end{equation}
where $[k]:=(q^k-1)/(q-1)$ is the $q$-number.
One can easily check that
\[
\qbinom{n}{m}=\prod ^{m-1}_{i=0}\frac{[n-i]}{[m-i]}.
\]
Assume $m<n$.
The set of all $m$-dimensional
subspaces of $V$ is denoted by $\qbinom{V}{m}$.
The set $\qbinom{V}{m}$ together with the nonempty relations
\[
R_i=\set{(W_1,W_2)\in {\textstyle \qbinom{V}{m} \times \qbinom{V}{m}} \mid \dim W_1\cap W_2=m-i}
\]
is an
$m\wedge (n-m)$-class symmetric association scheme called the
\emph{Grassmann scheme} $\mathrm{Gr}_q(n,m)$.
It is well known that Grassmann schemes are $P$- and $Q$-polynomial association schemes.
The size of $\qbinom{V}{m}$ is $\qbinom{n}{m}$.

The first eigenmatrix $P=(P^{\mathrm{Gr}_q(n,m)}_i(j))$ of the Grassmann scheme
$\mathrm{Gr}_q(n,m)$ is given by the \emph{generalized Eberlein polynomials} $E_i(n,m;q;j)$ (see Delsarte~\cite{Delsarte1976paa}), namely
\begin{eqnarray*}
P^{\mathrm{Gr}_q(n,m)}_i(j)&=&E_i(n,m;q;j)\\
&=&\sum^{i}_{u=0}(-1)^{i-u}q^{u j+\binom{i-u}{2}}\qbinom{m-u}{m-i}
\qbinom{m-j}{u}\qbinom{n-m+u-j}{u}.
\end{eqnarray*}
If either $i$ or $j$ is outside of $\{0,1,\ldots ,m\wedge (n-m)\}$,
then we define $E_i(n,m;q;j)=0$.
Furthermore, since the valencies $k^{\mathrm{Gr}_q(n,m)}_i$ and the multiplicities $m^{\mathrm{Gr}_q(n,m)}_i$
of the Grassmann scheme $\mathrm{Gr}_q(n,m)$ are given as
\begin{equation}
    \label{eq:Grassmann_k}
    k^{\mathrm{Gr}_q(n,m)}_i=q^{i^2}\qbinom{n-m}{i}\qbinom{m}{i}   
\end{equation}
and
\begin{equation}
    \label{eq:Grassmann_m}
    m^{\mathrm{Gr}_q(n,m)}_i=\qbinom{n}{i}-\qbinom{n}{i-1},
\end{equation}
respectively 
(cf.~\cite{BBIT2021,BI1984,BCN1989}).
The second eigenmatrix $Q=(Q^{\mathrm{Gr}_q(n,m)}_i(j))$
of the Grassmann scheme
$\mathrm{Gr}_q(n,m)$
is given by 
the \emph{$q$-Hahn polynomials}
$Q_i(n,m;q;j)$
as follows;
\begin{align*}
    Q^{\mathrm{Gr}_q(n,m)}_i(j) =& Q_i(n,m;q;j)\\
    =&\frac{\qbinom{n}{i}-\qbinom{n}{i-1}}
    {q^{j^2}\qbinom{n-m}{j}\qbinom{m}{j}}
    \sum^{j}_{u=0}(-1)^{j-u}q^{u i+\binom{j-u}{2}}\qbinom{m-u}{m-j}
    \qbinom{m-i}{u}\qbinom{n-m+u-i}{u}.
\end{align*}
If either $i$ or $j$ is outside $\{0,1,\ldots ,m\wedge (n-m)\}$,
then we define $Q_i(n,m;q;j)=0$.
It is known that
\[Q_{1}(n,m;q;j)=h^\ast(n,m) 
\left(\frac{[n-m][m]}{[n]}+[-j]\right),\]
where $h^\ast(n,m)=\frac{q[n][n-1]}{[n-m][m]}$
(cf.~\cite{BBIT2021}).
Hence, we have
\begin{equation}
    \label{eq:q^-x}
    [-j]=
    \frac{Q_{1}(n,m;q;j)}{h^\ast(n,m)}-\frac{[n-m][m]}{[n]}.    
\end{equation}
Also, we know
\begin{equation}
\label{eq:Grassmann_q}
q^\betaalt_{1j}(\mathrm{Gr}_q(n-i,m-i))=
\begin{cases}
h^\ast(n-i,m-i)B^\ast_{j-1}(n-i,m-i) & \text{if $\betaalt=j-1$;}\\
h^\ast(n-i,m-i)A^\ast_{j}(n-i,m-i) & \text{if $\betaalt=j$;}\\
h^\ast(n-i,m-i)C^\ast_{j+1}(n-i,m-i) & \text{if $\betaalt=j+1$;}\\
0 & \text{otherwise,}
\end{cases}
\end{equation}
where
\begin{align*}
    B^\ast_j(n-i,m-i)&=\frac{[m-i-j][n-i-j+1][n-m-j]}{q^j [n-i-2j][n-i-2j+1]},\\
    C^\ast_j(n-i,m-i)&=\frac{[j][n-m-j+1][m-i-j+1]}{q^j [n-i-2j+2][n-i-2j+1]},\\
    A^\ast_j(n-i,m-i)&=\frac{[n-m][m-i]}{[n-i]}-B^\ast_j(n-i,m-i)-C^\ast_j(n-i,m-i).
\end{align*}

\subsection{Association schemes based on attenuated spaces}
\label{sec:ASattenuated}

Let us recall the definition of the association schemes obtained from attenuated spaces.
Fix an $l$-dimensional subspace $W$ of the $(n+l)$-dimensional vector space $\F_q^{n+l}$
over $\F_q$.
The corresponding \emph{attenuated space} associated with $\F_q^{n+l}$ and $W$
is the collection of all subspaces of $\F_q^{n+l}$ intersecting trivially with $W$.
For a positive integer $m$ with $m\le n$,
let $X$ be the set of $m$-dimensional subspaces of the attenuated space associated with $\F_q^{n+l}$ and $W$.
Let 
\[
\mathcal{D}:=\{(i,j)\mid 0\le i \le m\wedge (n-m),\  0\le j\le (m-i)\wedge l\}
\]
and $\mathcal{R}\colon X \times X \to \mathcal{D}$ is defined by
$\mathcal{R}(V,V')=(i,j)$ if
\[
\dim V/W\cap V'/W=m-i \  \text{and}\  \dim V\cap V'=(m-i)-j,
\]
where $V/W$ stands for $(V+W)/W$ simply.
Then $\mathfrak{X}=(X,\mathcal{R})$ is a symmetric association scheme,
and called the \emph{association scheme obtained from the attenuated space} associated with $\F_q^{n+l}$ and $W$. 
For details of the association schemes obtained from attenuated spaces, 
see Bernard et al.~\cite{bi}, 
Wang-Guo-Li~\cite{WGL2010}
or Kurihara~\cite{Kurihara2013}.
The association scheme $\mathfrak{X}$
is a common generalization of the Grassmann scheme $\mathrm{Gr}_q(n,m)$
and the bilinear forms scheme $H_q(n,l)$.
In fact, if $l=0$, then the association scheme
$\mathfrak{X}$ is isomorphic to
the Grassmann scheme $\mathrm{Gr}_q(n,m)$ and if $m=n$,
then the association scheme $\mathfrak{X}$
is isomorphic to the bilinear forms scheme $H_q(n,l)$.
Moreover, the association scheme
$\mathfrak{X}$ is also a $q$-analogue of
the nonbinary Johnson scheme.
Bernard et al.~\cite{bi} proved that in the case of $l\ge m$,
$\mathfrak{X}$ becomes bivariate $P$-polynomial
of type $(1,0)$ on the domain $\mathcal{D}$.
Subsequently, Bannai et al.~\cite{BKZZ} showed that $\mathfrak{X}$ is a bivariate $P$-polynomial association scheme with respect to $\le_\mathrm{grlex}$ in the sense of Definition~\ref{df:abPpoly} even if $l<m$.

The cardinality of $\mathfrak{X}=(X,\mathcal{R})$ is given by
\begin{equation}    
\label{eq:attenuated_card}
|X|=q^{ml}\qbinom{n}{m}.
\end{equation}
It is known that the indices of the primitive idempotents of $\mathfrak{X}$ are also labeled by $\mathcal{D}$.
By \cite{Kurihara2013}, the explicit expressions of the first and second eigenmatrices of $\mathfrak{X}$ are given by 
\begin{equation}
    \label{eq:attenuated_P}
    P_{ij}(xy)=
    q^{j l}K_{i}(m-j,l;q;x)E_j(n-x,m-x;q;y)
    \end{equation}
    and
    \begin{equation}
    \label{eq:attenuated_Q}
    Q_{ij}(xy)=
    \frac{\qbinom{n}{m}}{\qbinom{n-i}{m-i}}K_i(m-y,l;q;x)
    Q_{j}(n-i,m-i;q;y).
    \end{equation}    
Moreover, the valencies $k_{ij}$ and the multiplicities $m_{ij}$ of $\mathfrak{X}$ are given by
\begin{align}
    k_{ij}&=q^{j l}k_i^{H_q(m-j,l)}k_j^{\mathrm{Gr}_q(n,m)}; \label{eq:attenuated_k}\\
    m_{ij}&=\frac{\qbinom{n}{m}}{\qbinom{n-i}{m-i}} m^{H_q(m,l)}_i m^{\mathrm{Gr}_q(n-i,m-i)}_j. \label{eq:attenuated_m}
\end{align}

In concluding this subsection,
we give some formulas and a lemma that will be used in the next sections.
First, the following formula may be easily verified:
\begin{align}
    [a]-[b] &= q^b [a-b]\  \text{for $a>b$};\label{iden1}\\
    \qbinom{N}{r} & =\frac{[N-r+1]}{[r]}\qbinom{N}{r-1};\label{iden2}\\
    \qbinom{N}{r} & =\frac{[N]}{[r]}\qbinom{N-1}{r-1};\label{iden3}\\
    \qbinom{N}{r} & =\frac{[N]}{[N-r]}\qbinom{N-1}{r};\label{iden4}\\
    \qbinom{N}{r}-\qbinom{N}{r-1} & 
    =\left(1-\frac{[r]}{[N-r+1]}\right)\qbinom{N}{r}
    =q^r\frac{[N-2r+1]}{[N-r+1]}\qbinom{N}{r}. \label{iden5}
\end{align}
The following lemma is a $q$-analog of Lemma~\ref{lem:Hahn}.
The proof of the lemma is also similar to \cite{biQ}.
\begin{lem}
\label{lem:qHahn} Let $p,N$ be integers such that $0<p<N$.
For $0\le x \le (p-1)\wedge (N-p)$
and $0\le r \le p\wedge (N-p)$, we have
\begin{equation}
\label{eq:q-Crampe}
\textstyle{
Q_r(N,p;q;x)
=\frac{[N]}{[p]} 
\left( 
\frac{[p-r]}{[N-2r]}Q_r(N-1,p-1;q;x)
+q^{p-r+1}\frac{[N-p-r+1]}{[N-2r+2]}Q_{r-1}(N-1,p-1;q;x)
\right).
}
\end{equation}
\end{lem}
Although \eqref{eq:q-Crampe} seems not to hold for some values of $r$,
it can be justified by interpreting it as follows.
For $r=0$, we regard $Q_{-1}(N-1,p-1;q;x)$ as zero,
i.e., \eqref{eq:q-Crampe} is written by
\[
    Q_0(N,p;q;x) = Q_0(N-1,p-1;q;x).
\]
In fact, both $Q_0(N,p;q;x)$ and $Q_0(N-1,p-1;q;x)$ become 1, hence the above equation is correct. 
For $r=p$, we regard $Q_p(N-1,p-1;q;x)$ as zero,
i.e., \eqref{eq:q-Crampe} is written by
\[
    Q_p(N,p;q;x)=q\frac{[N][N-2p+1]}{[p][N-2p+2]} Q_{p-1}(N-1,p-1;q;x).
\]
In fact, by $Q_p(N,p;x)=(-1)^x q^{p-\frac{x^2+x}{2}}\frac{[N-2p+1]}{[N-p+1]} \frac{\qbinom{N}{p}}{\qbinom{N-p}{x}}$, the above equation is justified.

\begin{proof}[Proof of Lemma~\ref{lem:qHahn}]

Using \eqref{iden4} and \eqref{iden3}, we have
\begin{align*}
Q_r(N,p;q;x)
=&
\frac{\qbinom{N}{r}-\qbinom{N}{r-1}}
{q^{x^2}\qbinom{N-p}{x}\qbinom{p}{x}}
\sum^{x}_{\ell=0}(-1)^{x-\ell}q^{\ell r+\binom{x-\ell}{2}}
\qbinom{p-\ell}{p-x}
\qbinom{p-r}{\ell}\qbinom{N-p+\ell-r}{\ell}\\
=&
\frac{1}
{[p]q^{x^2}\qbinom{N-p}{x}\qbinom{p-1}{x}}
\sum^{x}_{\ell=0}(-1)^{x-\ell}q^{\ell r+\binom{x-\ell}{2}}
\qbinom{p-1-\ell }{p-1-x}\\
&\times
[p-\ell]\left(\qbinom{N}{r}-\qbinom{N}{r-1}\right)
\qbinom{p-r}{\ell}\qbinom{N-p+\ell-r}{\ell}.
\end{align*}
Let us focus on the factor
$D=[p-\ell]\left(\qbinom{N}{r}-\qbinom{N}{r-1}\right)
\qbinom{p-r}{\ell}\qbinom{N-p+\ell-r}{\ell}$
in each summand and rewrite it using \eqref{iden5} and \eqref{iden1} 
as follows:
\begin{align}
D
=&
[p-\ell]\left(1-\frac{[r]}{[N-r+1]}\right)\qbinom{N}{r}
\qbinom{p-r}{\ell}\qbinom{N-p+\ell-r}{\ell} \notag\\
=&
\left([p-\ell]-\frac{[p-\ell][r]}{[N-r+1]}\right)\qbinom{N}{r}
\qbinom{p-r}{\ell}\qbinom{N-p+\ell-r}{\ell}\notag\\
=&
\left(q^r [p-\ell-r] + [r]-\frac{[p-\ell][r]}{[N-r+1]}\right)
\qbinom{N}{r}
\qbinom{p-r}{\ell}\qbinom{N-p+\ell-r}{\ell} \notag\\
=&
q^r [p-\ell-r]
\qbinom{N}{r}\qbinom{p-r}{\ell}\qbinom{N-p+\ell-r}{\ell} \label{eq:D1}\\
&+
[r]\left(1-\frac{[p-\ell]}{[N-r+1]}\right)
\qbinom{N}{r}\qbinom{p-r}{\ell}\qbinom{N-p+\ell-r}{\ell}\label{eq:D2}.
\end{align}
Then by \eqref{iden4} and \eqref{iden5}, we have
\begin{align*}
\eqref{eq:D1}
=&
q^r [p-\ell-r]
\frac{[N]}{[N-r]}
\qbinom{N-1}{r}\frac{[p-r]}{[p-r-\ell]}\qbinom{p-r-1}{\ell}\qbinom{N-p+\ell-r}{\ell}\\
=&
\frac{[N][p-r]}{[N-2r]}
\left(\qbinom{N-1}{r}-\qbinom{N-1}{r-1}\right)
\qbinom{p-r-1}{\ell}\qbinom{N-p+\ell-r}{\ell}
\end{align*}
and by \eqref{iden1}, \eqref{iden2}, \eqref{iden3} and \eqref{iden5},
we have
\begin{align*}
\eqref{eq:D2}
=&
[r]q^{p-\ell}\frac{[N-r+1-p+\ell]}{[N-r+1]}
\frac{[N]}{[r]}
\qbinom{N-1}{r-1}\qbinom{p-r}{\ell}
\frac{[N-p-r+1]}{[N-p+\ell-r+1]}\qbinom{N-p+\ell-r+1}{\ell}\\
=&
q^{p-\ell-r+1}[N]\frac{[N-p-r+1]}{[N-2r+2]}
\left(\qbinom{N-1}{r-1}-\qbinom{N-1}{r-2}\right)
\qbinom{p-r}{\ell}
\qbinom{N-p+\ell-r+1}{\ell}.
\end{align*}
Therefore, we have
\begin{align*}
&Q_r(N,p;q;x)\\
=&
\frac{1}
{[p]q^{x^2}\qbinom{N-p}{x}\qbinom{p-1}{x}}
\sum^{x}_{\ell=0}(-1)^{x-\ell}q^{\ell r+\binom{x-\ell}{2}}
\qbinom{p-1-\ell }{p-1-x}\\
&\times
\left(
    \frac{[N][p-r]}{[N-2r]}
    \left(\qbinom{N-1}{r}-\qbinom{N-1}{r-1}\right)
    \qbinom{p-r-1}{\ell}\qbinom{N-p+\ell-r}{\ell}
\right.\\
&+
\left.
    q^{p-\ell-r+1}[N]\frac{[N-p-r+1]}{[N-2r+2]}
    \left(\qbinom{N-1}{r-1}-\qbinom{N-1}{r-2}\right)
    \qbinom{p-r}{\ell}
    \qbinom{N-p+\ell-r+1}{\ell}
\right)\\
=&
\textstyle
\frac{[N]}{[p]}
\left(
\frac{[p-r]}{[N-2r]}
\frac{\qbinom{N-1}{r}-\qbinom{N-1}{r-1}}
{q^{x^2}\qbinom{N-p}{x}\qbinom{p-1}{x}}
\sum^{x}_{\ell=0}(-1)^{x-\ell}q^{\ell r+\binom{x-\ell}{2}}
\qbinom{p-1-\ell }{p-1-x}
\qbinom{p-r-1}{\ell}\qbinom{N-p+\ell-r}{\ell}
\right.\\
&+
\textstyle
\left.
q^{p-r+1}
\frac{[N-p-r+1]}{[N-2r+2]}
\frac{\qbinom{N-1}{r-1}-\qbinom{N-1}{r-2}}
{q^{x^2}\qbinom{N-p}{x}\qbinom{p-1}{x}}
\sum^{x}_{\ell=0}(-1)^{x-\ell}q^{\ell (r-1)+\binom{x-\ell}{2}}
\qbinom{p-1-\ell }{p-1-x}
\qbinom{p-r}{\ell}
\qbinom{N-p+\ell-r+1}{\ell}
\right)\\
=&
\frac{[N]}{[p]} 
\left( 
\frac{[p-r]}{[N-2r]}Q_r(N-1,p-1;q;x)
+q^{p-r+1}\frac{[N-p-r+1]}{[N-2r+2]}Q_{r-1}(N-1,p-1;q;x).
\right)
\end{align*}

\end{proof}

\section{Proof of Theorem~\ref{thm:main2}}
\label{sec:proofThm2}

In order to prove Theorem~\ref{thm:main2},
it is sufficient to prove the following proposition from Proposition~\ref{prop:Q-TFAE}.

\begin{prop}
\label{prop:attenuatedQ}
The Krein numbers of the association scheme obtained from the attenuated space $\mathfrak{X}$ satisfy the following:
\begin{enumerate}[label=$(\roman*)$]
\item $q^{\alphaalt \betaalt}_{10,ij} \neq 0$ implies
\[ 
(\alphaalt,\betaalt) \in \set{(i+1,j),(i+1,j-1),(i,j+1),(i,j),(i,j-1),(i-1,j+1),(i-1,j)}
\cap \mathcal{D},
\]
i.e., $(\alphaalt,\betaalt)\le_{\mathrm{grlex}} (i,j)+(1,0)$.
The exact values of $q^{\alphaalt \betaalt}_{10,ij}$ are given by
\eqref{eq:attenuated_Krein_1},
\eqref{eq:attenuated_Krein_2},
\eqref{eq:attenuated_Krein_3},
\eqref{eq:attenuated_Krein_4},
\eqref{eq:attenuated_Krein_5}
and
\eqref{eq:attenuated_Krein_6}.
Moreover, if $(i+1,j)=(i,j)+(1,0)\in \mathcal{D}$,
then $q^{i+1,j}_{10,ij} \neq 0$ holds.

\item $q^{\alphaalt \betaalt}_{01,ij} \neq 0$ implies
\[ 
(\alphaalt,\betaalt) \in \set{(i,j+1),(i,j),(i,j-1)}
\cap \mathcal{D},
\]
i.e., $(\alphaalt,\betaalt)\le_{\mathrm{grlex}} (i,j)+(0,1)$.
The exact values of $q^{\alphaalt \betaalt}_{01,ij}$ are given by
\eqref{eq:attenuated_Krein_011},
\eqref{eq:attenuated_Krein_012}
and
\eqref{eq:attenuated_Krein_013}.
Moreover, if $(i,j+1)=(i,j)+(0,1)\in \mathcal{D}$,
then $q^{i,j+1}_{01,ij} \neq 0$ holds.
\end{enumerate}
\end{prop}

\subsection{Krein numbers $q^{\alphaalt \betaalt}_{10,ij}$}
To prove (i) of Proposition~\ref{prop:attenuatedQ},
we compute the Krein numbers $q^{\alphaalt \betaalt}_{10,ij}$
of $\mathfrak{X}$
using the formula \eqref{eq:Krein2}.
Note that $\mathfrak{X}$ is symmetric, so there is no need to consider complex conjugates.
By \eqref{eq:attenuated_card},
\eqref{eq:attenuated_k},
\eqref{eq:attenuated_Q} and
\eqref{eq:attenuated_m},
we have
\begin{align*}
q^{\alphaalt \betaalt}_{10,ij}
=&
\frac{1}{|X| m_{\alphaalt \betaalt}}
\sum_{y=0}^{m \wedge (n-m)}\sum_{x=0}^{(m-y)\wedge l}
k_{xy} Q_{10}(xy)Q_{ij}(xy)Q_{\alphaalt \betaalt}(xy)\\
=&
\frac{[n]}{
[m]\qbinom{n-i}{m-i}
m^{H_q(m,l)}_\alphaalt m^{\mathrm{Gr}_q(n-\alphaalt,m-\alphaalt)}_\betaalt}
\sum_{y=0}^{m \wedge (n-m)}
k_y^{\mathrm{Gr}_q(n,m)}
Q_{j}(n-i,m-i;q;y)
Q_{\betaalt}(n-\alphaalt,m-\alphaalt;q;y)\\
&\times
\frac{1}{q^{(m-y) l}}
\sum_{x=0}^{(m-y)\wedge l}
k_x^{H_q(m-y,l)}
K_1(m-y,l;q;x)K_i(m-y,l;q;x)K_\alphaalt(m-y,l;q;x).
\end{align*}
Applying \eqref{eq:Krein2} to the bilinear forms scheme $H_q(k-y,l)$,
we have
\begin{align*}
q^{\alphaalt \betaalt}_{10,ij}
=&
\frac{[n]}{
[m]\qbinom{n-i}{m-i}
m^{H_q(m,l)}_\alphaalt m^{\mathrm{Gr}_q(n-\alphaalt,m-\alphaalt)}_\betaalt}
\sum_{y=0}^{m \wedge (n-m)}
k_y^{\mathrm{Gr}_q(n,m)}
Q_{j}(n-i,m-i;q;y)
Q_{\betaalt}(n-\alphaalt,m-\alphaalt;q;y)\\
&\times
m^{H_q(m-y,l)}_\alphaalt q_{1i}^\alphaalt (H_q(m-y,l)).
\end{align*}
Since one can check
$  m^{H_q(m-y,l)}_\alphaalt
=
m^{H_q(m,l)}_\alphaalt \qbinom{m-y}{\alphaalt}/\qbinom{m}{\alphaalt}$,
we have
\begin{align}
q^{\alphaalt \betaalt}_{10,ij}
=&
\frac{[n]}{
[m]\qbinom{n-i}{m-i}
\qbinom{m}{\alphaalt} m^{\mathrm{Gr}_q(n-\alphaalt,m-\alphaalt)}_\betaalt}
\sum_{y=0}^{m \wedge (n-m)}
k_y^{\mathrm{Gr}_q(n,m)}
Q_{j}(n-i,m-i;q;y)
Q_{\betaalt}(n-\alphaalt,m-\alphaalt;q;y)\notag\\
&\times 
\qbinom{m-y}{\alphaalt} q_{1i}^\alphaalt (H_q(m-y,l)).
\label{eq:attenuated_q^a_1i}
\end{align}    
According to \eqref{eq:bilinear_p_q},
$\alphaalt$ must be $i-1,i$ or $i+1$
In the below argument, we will use
\begin{equation}
    \label{eq:qbinom_rel}
    \qbinom{m}{y}\qbinom{m-y}{i}=\qbinom{m}{i}\qbinom{m-i}{y}.
\end{equation}

When the case $\alphaalt=i+1$ in \eqref{eq:attenuated_q^a_1i},
by \eqref{eq:Grassmann_k}, \eqref{eq:bilinear_p_q} and \eqref{eq:qbinom_rel}
we have
\begin{align*}
q^{i+1, \betaalt}_{10,ij}
=&
\frac{q^i [i+1][n]}{
[m]\qbinom{n-i}{m-i} m^{\mathrm{Gr}_q(n-i-1,m-i-1)}_\betaalt}\\
&\times
\sum_{y=0}^{m \wedge (n-m)}
k_y^{\mathrm{Gr}_q(n-i-1,m-i-1)}
Q_{j}(n-i,m-i;q;y)
Q_{\betaalt}(n-i-1,m-i-1;q;y).
\end{align*}
Since $Q_{\betaalt}(n-i-1,m-i-1;q;y)=0$ if $y>(m-i-1)\wedge ((n-i-1)-(m-i-1))$,
the range of $y$ in the above summation will be $0\le y \le (m-i-1)\wedge (n-m)$.
From \eqref{eq:q-Crampe},
\begin{align*}
    &Q_j(n-i,m-i;q;y)=\frac{[n-i]}{[m-i]}\\
    &\textstyle \times\left( 
    \frac{[m-i-j]}{[n-i-2j]}Q_j(n-i-1,m-i-1;q;y)
    +q^{m-i-j+1}\frac{[n-m-j+1]}{[n-i-2j+2]}Q_{j-1}(n-i-1,m-i-1;q;y)
    \right)
\end{align*}
holds.
Then the orthogonal relation \eqref{eq:Krein3}
for $\mathrm{Gr}_q(n-i-1,m-i-1)$ implies
\begin{align*}
q^{i+1, \betaalt}_{10,ij}
=&
\frac{q^i [i+1][n][n-i]}{
[m][m-i]\qbinom{n-i}{m-i} m^{\mathrm{Gr}_q(n-i-1,m-i-1)}_\betaalt}
|\mathrm{Gr}_q(n-i-1,m-i-1)|m^{\mathrm{Gr}_q(n-i-1,m-i-1)}_\betaalt\\
&\times
\left( \frac{[m-i-j]}{[n-i-2j]}\delta_{\betaalt ,j}+q^{m-i-j+1}\frac{[n-m-j+1]}{[n-i-2j+2]}\delta_{\betaalt ,j-1}\right)\\
=&
\frac{q^i [i+1][n]}{[m]}
\left( \frac{[m-i-j]}{[n-i-2j]}\delta_{\betaalt ,j}+q^{m-i-j+1}\frac{[n-m-j+1]}{[n-i-2j+2]}\delta_{\betaalt ,j-1}\right).
\end{align*}
In the last line, we used
\[
|\mathrm{Gr}_q(n-i-1,m-i-1)|=
\qbinom{n-i-1}{m-i-1}=
\qbinom{n-i}{m-i}\frac{[m-i]}{[n-i]}.
\]
Thus, $\betaalt$ must be $j$ or $j-1$, and we have
\begin{equation}
\label{eq:attenuated_Krein_1}
q^{i+1,j}_{10,ij} = 
q^i
\frac{[i+1][n][m-i-j]}{[m][n-i-2j]}
\quad
\text{and}
\quad
q^{i+1,j-1}_{10,ij} = 
q^{m-j+1}
\frac{[i+1][n][n-m-j+1]}{[m][n-i-2j+2]}.
\end{equation}

When the case $\alphaalt=i-1$ in \eqref{eq:attenuated_q^a_1i},
by 
\[
\qbinom{m}{i-1}=\frac{[i]}{[m-i+1]}\qbinom{m}{i}  
\quad  \text{and}\quad  
\qbinom{m-y}{i-1}= \frac{[i]}{[m-y-i+1]}\qbinom{m-y}{i},
\]
\eqref{eq:qbinom_rel} and \eqref{eq:Grassmann_k},
we have
\begin{align*}
q^{i-1, \betaalt}_{10,ij}
=&
q^{2i-2}(q-1)
\frac{[l-i+1][n][m-i+1]}{
[m]\qbinom{n-i}{m-i}
m^{\mathrm{Gr}_q(n-i+1,m-i+1)}_\betaalt}\\
&\times
\sum_{y=0}^{m \wedge (n-m)}
k^{\mathrm{Gr}_q(n-i,m-i)}_y
Q_{j}(n-i,m-i;q;y)
Q_{\betaalt}(n-i+1,m-i+1;q;y).
\end{align*}
Since $Q_{\betaalt}(n-i,m-i;q;y)=0$ if $y>(m-i)\wedge ((n-i)-(m-i))$,
the range of $y$ in the above summation will be $0\le y \le (m-i)\wedge (n-m)$.
From Lemma~\ref{lem:qHahn},
\begin{align*}
&Q_\betaalt(n-i+1,m-i+1;q;y)
\\
&=\frac{[n-i+1]}{[m-i+1]} 
\left( 
\frac{[m-i+1-\betaalt]}{[n-i+1-2\betaalt]}Q_\betaalt(n-i,m-i;q;y)
+q^{m-i-\betaalt+2}\frac{[n-m-\betaalt+1]}{[n-i-2\betaalt+3]}Q_{\betaalt-1}(n-i,m-i;q;y)
\right).    
\end{align*}
holds.
The orthogonal relation
\eqref{eq:Krein3} for $\mathrm{Gr}_q(n-i,m-i)$ implies
\begin{align*}
q^{i-1, \betaalt}_{10,ij}
=&
q^{2i-2}(q-1)
\frac{[l-i+1][n][n-i+1] m^{\mathrm{Gr}_q(n-i,m-i)}_j}{
[m] m^{\mathrm{Gr}_q(n-i+1,m-i+1)}_\betaalt}\\
&\times
\left( 
\frac{[m-i+1-\betaalt]}{[n-i+1-2\betaalt]}\delta_{j,\betaalt}
+q^{m-i-\betaalt+2}\frac{[n-m-\betaalt+1]}{[n-i-2\betaalt+3]}\delta_{j,\betaalt-1}
\right).
\end{align*}
Thus, $\betaalt$ must be $j$ or $j+1$.
Assume $\betaalt=j$.
Then the Krein number becomes
\[
q^{i-1, j}_{10,ij}
=
q^{2i-2}(q-1)
\frac{[l-i+1][n][n-i+1] m^{\mathrm{Gr}_q(n-i,m-i)}_j}{
[m] m^{\mathrm{Gr}_q(n-i+1,m-i+1)}_j}
\frac{[m-i+1-j]}{[n-i+1-2j]}.
\]
By \eqref{eq:Grassmann_m} and \eqref{iden5},
the above multiplicities can be calculated as
\[
m^{\mathrm{Gr}_q(n-i+1,m-i+1)}_j
=
q^j\frac{[n-i-2j+2][n-i+1]}{[n-i-j+2][n-i-j+1]}\qbinom{n-i}{j}
\]
and
\[
m^{\mathrm{Gr}_q(n-i,m-i)}_j
=
q^j\frac{[n-i-2j+1]}{[n-i-j+1]}\qbinom{n-i}{j}.
\]
Thus, we have
\begin{equation}
q^{i-1, j}_{10,ij}
=
q^{2i-2}(q-1)
\frac{[l-i+1][n][m-i-j+1][n-i-j+2]}{
[m][n-i-2j+2]}.
\label{eq:attenuated_Krein_2}    
\end{equation}
Assume $\betaalt=j+1$.
Then the Krein number becomes
\[
q^{i-1, j+1}_{10,ij}
=
q^{m+i-j-1}(q-1)
\frac{[l-i+1][n][n-i+1] m^{\mathrm{Gr}_q(n-i,m-i)}_j}{
[m] m^{\mathrm{Gr}_q(n-i+1,m-i+1)}_{j+1}}
\frac{[n-m-j]}{[n-i-2j+1]}.
\]
By \eqref{eq:Grassmann_m} and \eqref{iden5},
the multiplicity $m^{\mathrm{Gr}_q(n-i+1,m-i+1)}_{j+1}$ can be calculated as
\[
    m^{\mathrm{Gr}_q(n-i+1,m-i+1)}_{j+1}
    =
    q^{j+1}\frac{[n-i-2j][n-i+1]}{[n-i-j+1][j+1]}\qbinom{n-i}{j}.
\]
Thus, we have
\begin{equation}
q^{i-1, j+1}_{10,ij}
=
q^{m+i-j-2}(q-1)
\frac{[l-i+1][n] [j+1][n-m-j]}
{[m][n-i-2j]}.
\label{eq:attenuated_Krein_3}
\end{equation}

When the case $\alphaalt=i$ in \eqref{eq:attenuated_q^a_1i},
by \eqref{eq:bilinear_p_q}, \eqref{eq:Grassmann_k}, \eqref{eq:q^-x} and \eqref{eq:qbinom_rel}
we have
\begin{align*}
    q^{i \betaalt}_{10,ij}
=&
\frac{[n][i]}{
[m]\qbinom{n-i}{m-i}
m^{\mathrm{Gr}_q(n-i,m-i)}_\betaalt}
\sum_{y=0}^{m \wedge (n-m)}
k_y^{\mathrm{Gr}_q(n-i,m-i)}
Q_{j}(n-i,m-i;q;y)
Q_{\betaalt}(n-i,m-i;q;y)
\\
&\times 
(q^{m-y}+q^l-q^i-q^{i-1}-1)\\
=&
\frac{[n][i]}{
[m]\qbinom{n-i}{m-i}
m^{\mathrm{Gr}_q(n-i,m-i)}_\betaalt}
\sum_{y=0}^{m \wedge (n-m)}
k_y^{\mathrm{Gr}_q(n-i,m-i)}
Q_{j}(n-i,m-i;q;y)
Q_{\betaalt}(n-i,m-i;q;y)
\\
&\times 
\left(q^{m}(q-1)
\left(
\frac{Q_{1}(n-i,m-i;q;y)}{h^\ast (n-i,m-i)}-\frac{[n-m][m-i]}{[n-i]}
\right)
+q^m+q^l-q^i-q^{i-1}-1\right).
\end{align*}
Since $Q_{j}(n-i,m-i;q;y)=0$ if $y>(m-i)\wedge ((n-i)-(m-i))$,
the range of $y$ in the above summation will be $0\le y \le (m-i)\wedge (n-m)$.
The identities \eqref{eq:Krein2} and \eqref{eq:Krein3} for $\mathrm{Gr}_q(n-i,m-i)$ imply
\begin{align*}
q^{i \betaalt}_{10,ij}
=&
\frac{[n][i]}{[m]} 
\left(
\frac{q^{m}(q-1)}{h^\ast (n-i,m-i)}q^\betaalt_{1j}(\mathrm{Gr}_q(n-i,m-i)) \right.\\
&
\left. 
+(-q^{m}(q-1)\frac{[n-m][m-i]}{[n-i]}+q^m+q^l-q^i-q^{i-1}-1)\delta_{j\betaalt}
\right).  
\end{align*}
By \eqref{eq:Grassmann_q},
$\betaalt$ must be $j-1$, $j$ or $j+1$.
When $\betaalt=j-1$, we have
\begin{align}
    q^{i,j-1}_{10,ij}
    =&
    \frac{[n][i]}{[m]} 
    \frac{q^{m}(q-1)}{h^\ast (n-i,m-i)}
    h^\ast(n-i,m-i)B^\ast_{j-1}(n-i,m-i) \notag\\
    =&
    q^{m-j+1}(q-1)
    \frac{[n][i][m-i-j+1][n-m-j+1][n-i-j+2]}{[m][n-i-2j+2][n-i-2j+3]}.
    \label{eq:attenuated_Krein_4}
\end{align}
When $\betaalt=j+1$, we have
\begin{align}
    q^{i,j+1}_{10,ij}
    =&
    \frac{[n][i]}{[m]} 
    \frac{q^{m}(q-1)}{h^\ast (n-i,m-i)}
    h^\ast(n-i,m-i) C^\ast_{j+1}(n-i,m-i) \notag\\
    =&
    q^{m-j-1}(q-1)
    \frac{[n][i][j+1][n-m-j][m-i-j]}{[m][n-i-2j][n-i-2j-1]}.
    \label{eq:attenuated_Krein_5}
\end{align}
When $\betaalt=j$, we have
\begin{align}
q^{ij}_{10,ij}
=&
\frac{[n][i]}{[m]} 
\left(
q^m(q-1) A^\ast_{j}(n-i,m-i)
-q^{m}(q-1)\frac{[n-m][m-i]}{[n-i]}+q^m+q^l-q^i-q^{i-1}-1
\right) \notag\\
=&
\frac{[n][i]}{[m]} 
\left(
-q^m(q-1) (B^\ast_{j}(n-i,m-i)+ C^\ast_{j}(n-i,m-i))
+q^m+q^l-q^i-q^{i-1}-1
\right) \notag\\
=&
-q^{m-j}(q-1)\frac{[n][i]}{[m][n-i-2j+1]} \notag\\
&\times
\left(\frac{[m-i-j][n-i-j+1][n-m-j]}{[n-i-2j]}+ \frac{[j][n-m-j+1][m-i-j+1]}{[n-i-2j+2]}\right) \notag\\
&+(q^m+q^l-q^i-q^{i-1}-1)\frac{[n][i]}{[m]}.
\label{eq:attenuated_Krein_6}
\end{align}

\subsection{Krein numbers $q^{\alphaalt \betaalt}_{01,ij}$}
To prove (ii) of Proposition~\ref{prop:attenuatedQ},
we compute the Krein numbers $q^{\alphaalt \betaalt}_{01,ij}$ of $\mathfrak{X}$
using formula \eqref{eq:Krein2}.
By 
\eqref{eq:attenuated_card},
\eqref{eq:attenuated_k},
\eqref{eq:attenuated_Q} and
\eqref{eq:attenuated_m},
we have
\begin{align*}
q^{\alphaalt \betaalt}_{01,ij}
=&
\frac{1}{|X| m_{\alphaalt \betaalt}}
\sum_{y=0}^{m \wedge (n-m)}\sum_{x=0}^{(m-y)\wedge l}
k_{xy} Q_{01}(xy)Q_{ij}(xy)Q_{\alphaalt \betaalt}(xy)\\
=&\frac{1}{
q^{ml}\qbinom{n-i}{m-i}
 m^{H_q(m,l)}_\alphaalt m^{\mathrm{Gr}_q(n-\alphaalt,m-\alphaalt)}_\betaalt}
\sum_{y=0}^{m \wedge (n-m)}
q^{y l}k_y^{\mathrm{Gr}_q(n,m)}Q_{1}(n,m;q;y)Q_{j}(n-i,m-i;q;y)\\
&\times 
Q_{\betaalt}(n-\alphaalt,m-\alphaalt;q;y)
\sum_{x=0}^{(m-y)\wedge l}
k_x^{H_q(m-y,l)}K_i(m-y,l;q;x)K_\alphaalt(m-y,l;q;x).
\end{align*}
Applying \eqref{eq:Krein3} to $H_q(m-y,l)$,
we have
\[
\sum_{x=0}^{(m-y)\wedge l}
k_x^{H_q(m-y,l)}K_i(m-y,l;q;x)K_\alphaalt(m-y,l;q;x)
=
q^{(m-y)l}m^{H_q(m-y,l)}_\alphaalt \delta_{i,\alphaalt}.  
\]
Hence, $\alphaalt$ must be $i$ and 
by \eqref{eq:Grassmann_k}, \eqref{eq:bilinear_k_m} and \eqref{eq:qbinom_rel},
we have
\begin{align}
q^{i \betaalt}_{01,ij}
=&\frac{1}{
\qbinom{n-i}{m-i}
m^{\mathrm{Gr}_q(n-i,m-i)}_\betaalt} \notag\\
&\times \sum_{y=0}^{m \wedge (n-m)}
k_y^{\mathrm{Gr}_q(n-i,m-i)}
Q_{1}(n,m;q;y)Q_{j}(n-i,m-i;q;y)
Q_{\betaalt}(n-i,m-i;q;y) \label{eq:Q-up1}
\end{align}
Since $Q_{j}(n-i,m-i;q;y)=0$ if $y>(m-i)\wedge ((n-i)-(m-i))$,
the range of $y$ in the above summation will be $0\le y \le (m-i)\wedge (n-m)$.
To carry out the summation, we replace $Q_{1}(n,m;q;y)$ with another expression.
Similar way to \eqref{eq:theta-ast}, we have
In order to compute the sum,
we use the following calculation:
\begin{align}
    Q_{1}(n,m;q;x) 
    =&
    h^\ast(n,m) \left(\frac{Q_{1}(n-i,m-i;q;x)}{h^\ast(n-i,m-i)}+\frac{[n-m][m]}{[n]}-\frac{[n-m][m-i]}{[n-i]}
    \right). \label{eq:Q-down}
\end{align}
Thus 
\[
q^{i \betaalt}_{01,ij}
=
\frac{h^\ast(n,m)}{h^\ast(n-i,m-i)}q^\betaalt_{1j}(\mathrm{Gr}_q(n-i,m-i))
+h^\ast(n,m)\left(
\frac{[n-m][m]}{[n]}-\frac{[n-m][m-i]}{[n-i]}
\right)\delta_{j,\betaalt}
\]
holds.
By \eqref{eq:Grassmann_q},
$\betaalt$ must be $j-1$, $j$ or $j+1$.
When $\betaalt=j-1$, we have
\begin{align}
    q^{i,j-1}_{01,ij}
    =&
    \frac{h^\ast(n,m)}{h^\ast(n-i,m-i)}
    h^\ast(n-i,m-i)B^\ast_{j-1}(n-i,m-i) \notag\\
    =&
    \frac{[n][n-1][m-i-j+1][n-m-j+1][n-i-j+2]}{q^{j-2}[n-m][m][n-i-2j+2][n-i-2j+3]}.
    \label{eq:attenuated_Krein_011}
\end{align}
When $\betaalt=j+1$, we have
\begin{align}
    q^{i,j+1}_{01,ij}
    =&
    \frac{h^\ast(n,m)}{h^\ast(n-i,m-i)}
    h^\ast(n-i,m-i)C^\ast_{j+1}(n-i,m-i) \notag\\
    =&
    \frac{[n][n-1][j+1][n-m-j][m-i-j]}{q^{j}[n-m][m][n-i-2j][n-i-2j-1]}.
    \label{eq:attenuated_Krein_012}
\end{align}
When $\betaalt=j$, we have
\begin{align}
q^{i,j}_{01,ij}
=&
\frac{h^\ast(n,m)}{h^\ast(n-i,m-i)}
h^\ast(n-i,m-i)A^\ast_{j}(n-i,m-i) \notag\\
&+h^\ast(n,m)\left(
\frac{[n-m][m]}{[n]}-\frac{[n-m][m-i]}{[n-i]}
\right)
\notag\\
=&
h^\ast(n,m)
\left(
    \frac{[n-m][m]}{[n]}-B^\ast_{j}(n-i,m-i)-C^\ast_{j}(n-i,m-i)
\right)\notag\\
=&
\textstyle{
q [n-1]
-
\frac{[n][n-1]}{q^{j-1}[n-m][m][n-i-2j+1]}
\left(
    \frac{[m-i-j][n-i-j+1][n-m-j]}{[n-i-2j]}
    +
    \frac{[j][n-m-j+1][m-i-j+1]}{[n-i-2j+2]}
\right).
}
\label{eq:attenuated_Krein_013}
\end{align}

\section{$A_M$-Leonard pairs}
\label{sec:multiLeonard}

Iliev-Terwilliger~\cite{IT2012} consider some multivariate $P$-polynomial 
(and/or $Q$-polynomial) association schemes from the viewpoint 
of root systems, in particular of type $A_n$ and possibly for 
other types. These are very special classes of more general multivariate 
$P$-polynomial (and/or $Q$-polynomial) association schemes 
we have considered. 
Recently, Cramp\'{e}-Zaimi~\cite{CrampeZaimi2023}
gave further progress with respect to the above theory of \cite{IT2012}.
Here we give the definition of the generalized Lenard pair, called the $A_M$-Leonard pair, described in \cite{CrampeZaimi2023}.
Let $\mathbb{F}$ denote a field and let $V$ denote a vector space over $\mathbb{F}$ with finite positive dimension.
Let $\End(V)$ denote the $\mathbb{F}$-algebra consisting of the $\mathbb{F}$-linear maps from $V$ to $V$.
Fix integers $M,N \ge 1$.
Let 
\[\mathcal{D}:=\{\alpha \in \N^M \mid |\alpha| \le N\}.\]
A pair of elements $(r_1,\ldots ,r_M )$ and $(r'_1,\ldots ,r'_M )$ in $\mathcal{D}$ will be called \emph{adjacent}
whenever $(r_1 - r'_1, \ldots ,r_M - r'_M )$ is a permutation of an element in
\[\{(0,0,\ldots ,0),(1, -1, 0, 0,\ldots ,0),(1, 0, 0,\ldots ,0),(-1, 0, 0,\ldots ,0)\}.\]

\begin{dfe}
\label{df:AMLeonard}
The pair $(H,\tilde{H})$ is an \emph{$A_M$-Leonard pair} on the domain $\mathcal{D}$ if the following statements are satisfied: 
\begin{enumerate}[label=(\roman*)]
\item $H$ is an $M$-dimensional subspace of $\End(V)$ whose elements are diagonalizable and mutually commute.
\item $\tilde{H}$ is an $M$-dimensional subspace of $\End(V)$ whose elements are diagonalizable and mutually commute.
\item There exists a bijection $\alpha \mapsto V_\alpha $ from $\mathcal{D}$ to the set of common eigenspaces of $H$ such that for all $\alpha\in \mathcal{D}$,
\[\tilde{H}V_\alpha \subset \sum_{\substack{\beta \in \mathcal{D}\\ \text{$\beta$ adj $\alpha$}}} V_\beta .\]
\item There exists a bijection $\alpha \mapsto \tilde{V}_\alpha $ from $\mathcal{D}$ to the set of common eigenspaces of $\tilde{H}$ such that for all $\alpha\in \mathcal{D}$,
\[H\tilde{V}_\alpha \subset \sum_{\substack{\beta \in \mathcal{D}\\ \text{$\beta$ adj $\alpha$}}} \tilde{V}_\beta .\]
\item There does not exist a subspace $W$ of $V$ such that $H W \subset W$,
$\tilde{H} W \subset W$, $W\neq 0$, $W\neq V$.
\item Each of $V_\alpha,\tilde{V}_\alpha$ has dimension 1 for $\alpha \in \mathcal{D}$.
\item There exists a nondegenerate symmetric bilinear form $\langle , \rangle$ on $V$ such that both
\begin{align*}
&\langle V_\alpha ,V_\beta \rangle =0 \quad \text{if $\alpha \neq \beta$ and $\alpha, \beta \in \mathcal{D}$},\\
&\langle \tilde{V}_\alpha ,\tilde{V}_\beta \rangle =0 \quad \text{if $\alpha \neq \beta$ and $\alpha, \beta \in \mathcal{D}$}.
\end{align*}
\end{enumerate}
\end{dfe}

Among multivariate polynomial association schemes, 
we define a class related to the $A_M$ Leonard pairs 
and show that it has the structure of $A_M$-Leonard pairs as follows.
\begin{dfe}
\label{df:AMmultivariate}
A symmetric association scheme $\mathfrak{X}$ is called an \emph{$A_M$ multivariate $P$- and $Q$-polynomial association scheme} 
on $\mathcal{D}=\{\alpha \in \N^M \mid |\alpha| \le N\}$
if the following conditions are satisfied:
\begin{enumerate}[label=(\roman*)]
\item $\mathfrak{X}$ is an $M$-variate $P$-polynomial association scheme on $\mathcal{D}$ for some monomial order $\le _1$ such that 
for $\alpha \in \mathcal{D}$ and $i=1,2,\ldots ,M$, if $p_{\epsilon_i,\alpha}^{\beta} \neq 0$ then $\beta$ is adjacent to $\alpha$;
\item $\mathfrak{X}$ is an $M$-variate $Q$-polynomial association scheme on $\mathcal{D}$ for some monomial order $\le _2$ such that 
for $\alpha \in \mathcal{D}$ and $i=1,2,\ldots ,M$, if $q_{\epsilon_i,\alpha}^{\beta} \neq 0$ then $\beta$ is adjacent to $\alpha$.
\end{enumerate}
\end{dfe}

By considering the principal module of the Terwilliger algebra of an association scheme, 
it can be shown that an $A_M$-multivariate $P$- and $Q$-polynomial association scheme has the structure of $A_M$-Leonard pairs
as Theorem~\ref{thm:AMLeonard}. 
We will now briefly review the principal module of the Terwilliger algebra. 
For details, please refer to Terwilliger~\cite{Terwilligeralg1,Terwilligeralg2,Terwilligeralg3}.

Let $\mathfrak{X} = (X, \{R_i\}_{i \in \mathcal{I}})$ be a general commutative association scheme. 
The complex vector space with the set $X$ of points is represented by $W = \C^{X}$. 
Suppose $x_0 \in X$ is a point of $\mathfrak{X}$ and for each $i \in \mathcal{I}$, $\Gamma_i(x_0) = \{x \mid (x_0,x) \in R_i\}$. 
The subspace of $W$ spanned by $\Gamma_i(x_0)$ is denoted as $W_i^*$.
Let $E_i^*$ be the projection matrix from $W$ to $W_i^*$.
In other words, $E_i^*$ is the diagonal matrix given by 
\begin{align}
E_i^*(x,y) = 
\begin{cases}
1, & \text{if } x=y \text{ and } (x_0,x) \in R_i, \\
0, & \text{otherwise.}
\end{cases}
\end{align}
For $i \in \mathcal{J}$, we define 
\begin{align}
A_i^* = \sum_{\alpha \in \mathcal{I}} Q_i(\alpha) E_\alpha^*.
\end{align}
The Bose-Mesner algebra $\mathfrak{A}$ is generated by $\{A_0, A_1, \ldots, A_d\}$ and the dual Bose-Mesner algebra $\mathfrak{A}^*$ with respect to $x_0$ is generated by $\{A_0^*, A_1^*, \ldots, A_d^*\}$. 
The \emph{Terwilliger algebra} $T = T(x_0)$ is generated by $\mathfrak{A}$ and $\mathfrak{A}^*$. 
Let $\bm{1} \in \C^X$ be the all-one vector. 
It is known that the subspace $\mathfrak{A} x_0 = \mathfrak{A}^* \bm{1}$ is $T$-invariant, and thus $\mathfrak{A}x_0$ is called the \emph{principal $T$-module}. 
Let $V = \mathfrak{A}x_0$, $v_i = E_i^* \bm{1} \in V_i^* = E_i^* V$, and $v_i^* = E_i x_0 \in V_i = E_i V$. 
Then one can check the following facts:
\begin{enumerate}[label=$(T\arabic*)$]
\item \label{enu:Ajv*i} $A_j v_i^* = P_j(i) v_i^*$ for $j \in \mathcal{I}$ and $i \in \mathcal{J}$;
\item \label{enu:A*jvi} $A_j^\ast v_i = Q_j(i) v_i$ for $i \in \mathcal{I}$ and $j \in \mathcal{J}$;  
\item \label{enu:dimV=1} $\dim V_i = 1$ for $i \in \mathcal{J}$;
\item \label{enu:dim*V=1} $\dim V_i^* = 1$ for $i \in \mathcal{I}$; 
\item \label{enu:Ajvi} $A_j v_i = \sum_{k \in \mathcal{I}} p_{j,i}^k v_k$;
\item \label{enu:A*jv*i} $A_j^* v_i^* = \sum_{k \in \mathcal{J}} q_{j,i}^k v_k^*$.
\end{enumerate}

\begin{thm}
\label{thm:AMLeonard}
An $A_M$ multivariate $P$- and $Q$-polynomial association scheme on $\mathcal{D}=\{\alpha \in \N^M \mid |\alpha| \le N\}$
has the structure of $A_M$-Leonard pairs on $\mathbb{F}=\C$.
\end{thm}

\begin{proof}

Let $H$ be the subspace spanned by $A_{\epsilon_1},A_{\epsilon_2},\ldots ,A_{\epsilon_M}$, 
and let $\tilde{H}$ be the subspace spanned by $A^\ast_{\epsilon_1},A^\ast_{\epsilon_2},\ldots ,A^\ast_{\epsilon_M}$. 
Then \ref{enu:Ajv*i} and \ref{enu:A*jvi} lead to (i) and (ii) of Definition~\ref{df:AMLeonard}.
Let $\tilde{V}_{\alpha} = V_\alpha^*$. 
By \ref{enu:dimV=1}, \ref{enu:dim*V=1}, \ref{enu:Ajvi} and \ref{enu:A*jv*i}, 
and the definition of $A_M$ multivariate $P$- and $Q$-polynomial, we know that (iii), (iv), and (vi) of Definition~\ref{df:AMLeonard} hold.
Since Definition~\ref{df:AMmultivariate} requires that an $A_M$ multivariate $P$- and $Q$-polynomial association scheme is symmetric, 
we know that (vii) of Definition~\ref{df:AMLeonard} holds as well. 

In the following, we will show that (v) holds. 
Now suppose there exists a subspace $U$ of $V$ such that $HU \subset U$, $\tilde{H} U \subset U$, $U \neq 0$. 
Take $0 \neq v = \sum_{\alpha \in \mathcal{D}} c_\alpha v_\alpha^* \in U$. 
In particular, there exists $\alpha_0 \in \mathcal{D}$ such that $c_{\alpha_0} \neq 0$. 
By the above argument $v_{\alpha_0}^* = \frac{1}{c_{\alpha_0}} E_{\alpha_0} v \in U$. 
Since $\mathfrak{X}$ is an $M$-multivariate $Q$-polynomial association scheme,
for $i=1,2,\ldots ,M$, we have $q_{\epsilon_i,\alpha_0}^{\alpha_0+\epsilon_i} \neq 0$ (resp. $q_{\epsilon_i,\alpha_0}^{\alpha_0-\epsilon_i} \neq 0$) 
whenever $\alpha_0+ \epsilon_i \in \mathcal{D}$ (resp. $\alpha_0- \epsilon_i \in \mathcal{D}$).
By \ref{enu:A*jv*i}, it follows that
\[
    v_{\alpha_0+\epsilon_i}^*  = \frac{1}{q_{\epsilon_i,\alpha_0}^{\alpha_0+\epsilon_i}}E_{\alpha_0 +\epsilon_i} A^\ast_{\epsilon_i} v^\ast_{\alpha_o}\in U
    \  \text{and}\  
    v_{\alpha_0-\epsilon_i}^*  = \frac{1}{q_{\epsilon_i,\alpha_0}^{\alpha_0-\epsilon_i}}E_{\alpha_0 -\epsilon_i} A^\ast_{\epsilon_i} v^\ast_{\alpha_o}\in U.
\]
By induction, we get $v_\alpha^* \in U$ for every $\alpha \in \mathcal{D}$. 
Therefore, $U = V$. 
\end{proof}

For suitable parameters,
we can show that
nonbinary Johnson schemes and association schemes obtained from attenuated space
are $A_2$ bivariate $P$- and $Q$-polynomial association schemes
by \cite{BKZZ} or \cite{bi} and Propositions~\ref{prop:nonbinaryQ} and \ref{prop:attenuatedQ}.
Hence, we have the following corollary.
\begin{cor}
\label{cor:multiLeonard}
If the domain $\mathcal{D}$ of a nonbinary Johnson scheme or an association scheme obtained from attenuated space
becomes an isosceles right triangle,
then the association scheme has the structure of $A_2$-Leonard pairs.
\end{cor}


\begin{Rem}
Our Corollary~\ref{cor:multiLeonard} is interesting because we get examples 
satisfying the conditions of Problem 7.1 in Iliev-Terwilliger~\cite{IT2012}
(namely, 
Definition 7.1 in the present paper), where $H$ and $\tilde{H}$ actually 
come from association schemes, i.e., not just at the linear algebra level.
We expect that similar results can be obtained for association schemes 
of (not necessarily maximal) isotropic subspaces, and also for 
generalized Johnson association schemes in the sense of Ceccherini-Silbersten et al.~\cite{CST2006}. We hope to come back to these problems in the near future.     
\end{Rem}

\vspace{2ex}

\noindent
{\bf{Acknowledgments}}

\vspace{2ex}

\noindent
The research of the second author is supported by JSPS, Japan KAKENHI Grant Number JP20K03623.
The third author thanks Kyoto University and Shanghai Jiao Tong University.
The research of the fourth author is supported by National Natural Science Foundation of China No.~11801353 and the Fundamental Research Funds for the Central Universities.

\bibliographystyle{abbrv}
\bibliography{ASbib}

\begin{thebibliography}{10}

\bibitem{BBIT2021}
E.~Bannai, E.~Bannai, T.~Ito, and R.~Tanaka.
\newblock {\em Algebraic Combinatorics}.
\newblock De Gruyter, Berlin, Boston, 2021.

\bibitem{BI1984}
E.~Bannai and T.~Ito.
\newblock {\em Algebraic combinatorics. {I}: Association Schemes}.
\newblock The Benjamin/Cummings Publishing Co., Inc., Menlo Park, CA, 1984.

\bibitem{BKZZ}
E.~Bannai, H.~Kurihara, D.~Zhao, and Y.~Zhu.
\newblock Multivariate {$P$}- and/or {$Q$}-polynomial association schemes.
\newblock {\em ArXiv: 2305.00707v2}, 2023.

\bibitem{bi}
P.~A. Bernard, N.~Cramp\'{e}, L.~P. d'Andecy, L.~Vinet, and M.~Zaimi.
\newblock Bivariate {$P$}-polynomial association schemes.
\newblock {\em ArXiv: 2212.10824}, 2022.

\bibitem{BCN1989}
A.~E. Brouwer, A.~M. Cohen, and A.~Neumaier.
\newblock {\em Distance-regular graphs}, volume~18 of {\em Ergebnisse der Mathematik und ihrer Grenzgebiete (3) [Results in Mathematics and Related Areas (3)]}.
\newblock Springer-Verlag, Berlin, 1989.

\bibitem{CST2006}
T.~Ceccherini-Silberstein, F.~Scarabotti, and F.~Tolli.
\newblock Trees, wreath products and finite {G}elfand pairs.
\newblock {\em Adv. Math.}, 206(2):503--537, 2006.

\bibitem{Cox}
D.~A. Cox, J.~Little, and D.~O'Shea.
\newblock {\em Ideals, varieties, and algorithms}.
\newblock Undergraduate Texts in Mathematics. Springer, Cham, fourth edition, 2015.
\newblock An introduction to computational algebraic geometry and commutative algebra.

\bibitem{biQ}
N.~Cramp\'{e}, L.~Vinet, M.~Zaimi, and X.~Zhang.
\newblock A bivariate {$Q$}-polynomial structure for the non-binary johnson scheme.
\newblock {\em J. Combinatorial Theory Ser. A}, 202:105829, 2024.

\bibitem{CrampeZaimi2023}
N.~Cramp\'{e} and M.~Zaimi.
\newblock Factorized {$A_2$}-leonard pair.
\newblock {\em arXiv:2312.08312}, 2023.

\bibitem{Delsarte1976paa}
P.~Delsarte.
\newblock Association schemes and {$t$}-designs in regular semilattices.
\newblock {\em J. Combinatorial Theory Ser. A}, 20(2):230--243, 1976.

\bibitem{Delsarte1978bfo}
P.~Delsarte.
\newblock Bilinear forms over a finite field, with applications to coding theory.
\newblock {\em J. Combin. Theory Ser. A}, 25(3):226--241, 1978.

\bibitem{GR2007}
G.~Gasper and M.~Rahman.
\newblock Some systems of multivariable orthogonal {$q$}-{R}acah polynomials.
\newblock {\em Ramanujan J.}, 13(1-3):389--405, 2007.

\bibitem{IT2012}
P.~Iliev and P.~Terwilliger.
\newblock The {R}ahman polynomials and the {L}ie algebra {$\mathfrak{sl}_3({\mathbb{C}})$}.
\newblock {\em Trans. Amer. Math. Soc.}, 364(8):4225--4238, 2012.

\bibitem{Kurihara2013}
H.~Kurihara.
\newblock Character tables of association schemes based on attenuated spaces.
\newblock {\em Ann. Comb.}, 17(3):525--541, 2013.

\bibitem{Leonard1982}
D.~A. Leonard.
\newblock Orthogonal polynomials, duality and association schemes.
\newblock {\em SIAM J. Math. Anal.}, 13(4):656--663, 1982.

\bibitem{MT2004}
H.~Mizukawa and H.~Tanaka.
\newblock {$(n+1,m+1)$}-hypergeometric functions associated to character algebras.
\newblock {\em Proc. Amer. Math. Soc.}, 132(9):2613--2618, 2004.

\bibitem{Stanton1980}
D.~Stanton.
\newblock Some {$q$}-{K}rawtchouk polynomials on {C}hevalley groups.
\newblock {\em Amer. J. Math.}, 102(4):625--662, 1980.

\bibitem{Stein2017}
I.~Stein.
\newblock The {L}ittlewood-{R}ichardson rule for wreath products with symmetric groups and the quiver of the category {$F\wr {\bf FI}_n$}.
\newblock {\em Comm. Algebra}, 45(5):2105--2126, 2017.

\bibitem{TAG1985}
H.~Tarnanen, M.~J. Aaltonen, and J.-M. Goethals.
\newblock On the nonbinary {J}ohnson scheme.
\newblock {\em European J. Combin.}, 6(3):279--285, 1985.

\bibitem{Terwilligeralg1}
P.~Terwilliger.
\newblock The subconstituent algebra of an association scheme. {I}.
\newblock {\em J. Algebraic Combin.}, 1(4):363--388, 1992.

\bibitem{Terwilligeralg2}
P.~Terwilliger.
\newblock The subconstituent algebra of an association scheme. {II}.
\newblock {\em J. Algebraic Combin.}, 2(1):73--103, 1993.

\bibitem{Terwilligeralg3}
P.~Terwilliger.
\newblock The subconstituent algebra of an association scheme. {III}.
\newblock {\em J. Algebraic Combin.}, 2(2):177--210, 1993.

\bibitem{Terwilliger2001}
P.~Terwilliger.
\newblock Two linear transformations each tridiagonal with respect to an eigenbasis of the other.
\newblock {\em Linear Algebra Appl.}, 330(1-3):149--203, 2001.

\bibitem{Terwilliger2021}
P.~Terwilliger.
\newblock Notes on the {L}eonard system classification.
\newblock {\em Graphs Combin.}, 37(5):1687--1748, 2021.

\bibitem{WGL2010}
K.~Wang, J.~Guo, and F.~Li.
\newblock Association schemes based on attenuated spaces.
\newblock {\em European J. Combin.}, 31(1):297--305, 2010.

\end{thebibliography}

\end{document}